\documentclass[a4paper]{amsart}
\usepackage{amsmath,amsthm,amssymb,latexsym,epic,bbm,comment,mathrsfs}
\usepackage{graphicx,enumerate,stmaryrd, xcolor, color}
\usepackage[all,2cell]{xy}
\xyoption{2cell}

\usepackage{soul}

\theoremstyle{plain}
\newtheorem{thm}{Theorem}
\newtheorem*{thm*}{Theorem}

\newtheorem{lem}[thm]{Lemma}
\newtheorem{prop}[thm]{Proposition}

\newtheorem{cor}[thm]{Corollary}
\newtheorem{df-prop}[thm]{Definition-Proposition}

\theoremstyle{definition}

\theoremstyle{remark}
\newtheorem{rem}[thm]{Remark}

\usepackage[all]{xy}
\usepackage[active]{srcltx}
\usepackage[parfill]{parskip}
\usepackage{enumerate}

\newcommand{\Hom}{\operatorname{Hom}}

\usepackage{hyperref}
\newcommand{\mc}{\mathcal}
\newcommand{\mf}{\mathfrak}
\newcommand{\C}{\mathbb C}

\newcommand{\oa}{{\bar 0}}
\newcommand{\ob}{{\bar 1}}
\def\gl{\mathfrak{gl}}
\newcommand{\g}{\mathfrak{g}}
\def\Ann{{\text{Ann}}}
\def\la{\lambda}

\def\ov{\overline}
\newcommand{\ch}{\mathrm{ch}}
\newcommand{\h}{\mathfrak{h}}

\newcommand{\Z}{{\mathbb Z}}

\newcommand{\redtext}[1]{\textcolor{red}{#1}}

\newcommand{\red}[1]{\redtext{ #1}}

\def\mod{\operatorname{-mod}\nolimits}
\def\Mod{\operatorname{-Mod}\nolimits}

\def\Hom{\operatorname{Hom}\nolimits}
\def\End{\operatorname{End}\nolimits}

\def\Ind{\operatorname{Ind}\nolimits}

\def\wt{\operatorname{wt}\nolimits}

\def\pr{\operatorname{pr}\nolimits}

\def\gl{\mathfrak{gl}}

\def\la{\lambda}

\def\ov{\overline}
\newcommand{\ad}{\mathrm{ad}}

\begin{document}
\title[Whittaker modules and finite SUSY $W$-algebras]{Whittaker Modules of Central Extensions of Takiff Superalgebras and Finite Supersymmetric $W$-algebras}

\author[Chen]{Chih-Whi Chen} \address{Department of Mathematics, National Central University, Chung-Li, Taiwan 32054 \\ National Center of Theoretical Sciences,
	Taipei, Taiwan 10617} \email{cwchen@math.ncu.edu.tw}

\author[Cheng]{Shun-Jen Cheng}
\address{Institute of Mathematics, Academia Sinica, Taipei, Taiwan 10617} \email{chengsj@math.sinica.edu.tw}

\author[Suh]{Uhi Rinn Suh} \address{Department of Mathematical Sciences and Research institute of Mathematics, Seoul National University,
	Gwanak-ro 1, Gwanak-gu, Seoul 08826, Korea} \email{uhrisu1@snu.ac.kr}
\date{}

\begin{abstract}
For a basic classical Lie superalgebra $\mf s$, let $\g$ be the central extension of the Takiff superalgebra $\mf s\otimes\Lambda(\theta)$, where $\theta$ is an odd indeterminate. We study the category of $\g$-Whittaker modules associated with a nilcharacter $\chi$ of $\g$ and show that it is equivalent to the category of $\mf s$-Whittaker modules associated with a nilcharacter of $\mf s$ determined by $\chi$. 
In the case when $\chi$ is regular, we obtain, as an application, an equivalence between the categories of modules over the supersymmetric finite $W$-algebras associated to the odd principal nilpotent element at non-critical levels and the category of the modules over the principal finite $W$-superalgebra associated to $\mf s$. Here, a supersymmetric finite $W$-algebra is conjecturally the Zhu algebra of a supersymmetric affine $W$-algebra.  This allows us to classify and construct irreducible representations of a principal finite supersymmetric $W$-algebra.
 \end{abstract}

\maketitle

\tableofcontents



\section{Introduction}

\subsection{Motivation}
Let $\mf s$ be a finite-dimensional Lie algebra and $\mf n$ be a nilradical of $\mf s$. Let $\chi:\mf n\rightarrow\C$ be a character and set $\mf n_\chi:=\{x-\chi(x)\mid x\in\mf n\}$. A finitely generated $U(\mf s)$-module on which $\mf n_\chi$ acts locally nilpotently, is called a {\em Whittaker module}. These modules have first been studied in the case when $\chi$ is regular, i.e., $\chi$ is nonzero on all simple root vectors, by Kostant \cite{Ko78} who showed that this category is equivalent to the category of modules over $Z(\mf s)$. In the case of arbitrary $\chi$, the corresponding  Whittaker modules have been systematically developed by McDowell \cite{Mc85} and Milicic-Soergel \cite{MS97}.  
Since then, considerable progress has been made in various contexts; see, e.g., \cite{B97, BM11, We11, MS14, AB21}  and references therein. Besides their own  interest and applications, Whittaker modules are an important ingredient in the study of so-called finite $W$-algebras. These associative algebras, defined for every nilpotent element $E\in\mf s$, have remarkable structure and representation theory. Due to their highly non-linear nature a direct approach to study them is very difficult. As it turns out, Whittaker modules provide a viable way to study the representation theory of $W$-algebras \cite{Skr, BGK08, Los12}.
 
While Whittaker modules for Lie algebras have attracted a lot of attention over the years, their analogues for Lie superalgebras were studied in detail only quite recently, see, e.g., \cite{BCW14, C21, CCM23, CC23_2, CC24, CW24}. In loc.~cit., various aspects of the category  of Whittaker modules over quasi-reductive Lie superalgebras have been investigated, including the connections with representation theory of principal finite $W$-superalgebras. 

In this paper we consider the following setup: Let $\mf s$ be a basic classical Lie superalgebra  with an even non-degenerate supersymmetric invariant bilinear  form $(\_|\_)$ and let $\Lambda(\theta)$ be the Grassmann superalgebra in one indeterminate $\theta$. We form the Takiff Lie superalgebra $\mf s\otimes \Lambda(\theta)$, and then consider its central extension $(\g, [\_,\_])$ by a one-dimensional center $\C z$ with Lie bracket: 
\begin{align}
&[s_1\otimes f_1,s_2\otimes f_2]:={(-1)^{p(f_1)p(s_2)}}[s_1,s_2]\otimes f_1 f_2+ (-1)^{ p(s_2)} (s_1|s_2)\frac{\partial f_1}{\partial\theta}\frac{\partial f_2}{\partial\theta}z, \label{eq::defofg:1}
\end{align} 
where $s_1,s_2\in \mf s$ are homogeneous elements and $f_1,f_2\in\Lambda(\theta)$. Here $p(\cdot)$ denotes the parity function.

For a nilradical $\mf n$ of $\mf s$, let $\hat{\mf n}:=\mf n\otimes\Lambda(\theta)$ and $\chi:\hat{\mf n}\rightarrow\C$ be a character. The main purpose of this paper is to study the Whittaker modules of $\g$, i.e., the category of $\g$-modules on which $\hat{\mf n}_\chi$ acts locally nilpotently. Note that, in the most degenerate case of $\chi=0$, this category of Whittaker modules reduces to the so-called ``thick'' category $\mc O$, and thus, the investigation in the present paper generalizes and includes some of the results in \cite{Ch93}, where, among others, finite-dimensional $\g$-modules were studied when $\mf s$ is a simple Lie algebra.

A motivation for studying the representation theory of $\g$ stems from the so-called supersymmetric (SUSY) $W$-algebras introduced in \cite{MR94} as an underlying vertex algebra of a certain SUSY conformal field theory. A special feature of a SUSY W-algebra is the existence of a special odd automorphism $D$ that assigns a super-partner to each element. A vertex algebra with such an odd automorphism $D$ is called a SUSY vertex algebra \cite{HK07} and SUSY W-algebra is a SUSY vertex algebra. 

In contrast to usual (affine) $W$-algebra, the data of a SUSY W-algebra are induced from a basic classical Lie superalgebra $\mf s$ and an odd nilpotent element $e$. Roughly speaking, the SUSY W-algebra is a SUSY BRST Hamiltonian reduction of the SUSY affine vertex algebra which is the chiralization of the universal enveloping algebra of $\g$ introduced above. Here, the indeterminate $\theta$ arises from the odd automorphism $D$ of the SUSY affine vertex algebra.
 The best understood case is when $e$ can be included in an $\mathfrak{osp}(1|2)$ subalgebra of $\mf s$ \cite{MRS21,Suh20,CS21} and 
$e$ is principal nilpotent. In this case, it is shown in \cite{GSS25} that the SUSY W-algebra is isomorphic to the usual $W$-algebra associated to the even principal nilpotent $E=\frac{1}{2}[e,e]$.

Recall that the finite $W$-algebra can be realized as the Zhu algebra of the corresponding $W$-algebra associated to an even nilpotent element of a simple Lie superalgebra \cite{DSK06, Gen24}.  This connection provides one of main physical motivations to study finite $W$-algebras, since there is 1-1 correspondence between positive energy irreducible modules of a vertex algebra and irreducible modules of its Zhu algebra \cite{Zhu96,FZ92}.

By the same reason, the Zhu algebra of a SUSY $W$-algebra is an interesting object to understand. Indeed, SUSY W-algebras have superconformal vectors \cite{Song24,SY23} which allow us to define their Zhu algebras. It is expected that the Zhu algebra of the SUSY $W$-algebra associated to the $\mf s$ and $e$ at level $c$ is the following:
\begin{equation} \label{eq:finie SUSY}
    \mc{SW}_c(\mf s,e):=(U^c(\g)\otimes _{U(\hat{\mathfrak{n}})}\mathbb{C}_{\chi^e})^{\text{ad}\hat{\mathfrak{n}}}.
\end{equation}
Here, $U^c(\g)$ is the quotient algebra of $U(\g)$ by the ideal generated by $z-c$ and $\C_{\chi^e}$ is one-dimensional $\hat{\mf n}$-module with character of $\hat{\mathfrak{n}}$ defined by $\chi^e(a)=(e|\frac{\partial}{\partial\theta} a)$ for $a\in \hat{\mathfrak{n}}$, where we identify the derivation $\frac{\partial}{\partial\theta}$ of $\Lambda(\theta)$ with a derivation of $\g$.

The motivation to understand the representation theory of $\mc{SW}_c(\mf s,e)$ in turn leads us to study Whittaker modules over $\g$ which we explain below.

\subsection{Description of results} \subsubsection{}  \label{sect::111}
From now on, let $\mf s$ be a basic classical Lie superalgebra  with an even non-degenerate supersymmetric invariant bilinear  form $(\_|\_)$ and let $\g$ be defined as in \eqref{eq::defofg:1}.

We may  identify $\mf s$ with the subalgebra $\mf s\otimes 1 \subset \mf g$.  Fix a triangular decomposition of $\mf s$ in the sense of \cite{Ma14, CCC21}:
\begin{align}
&\mf s=\mf n^- \oplus \h \oplus \mf n,\label{eq::tri0}
\end{align} with Cartan subalgebra $\h$, nilradical $\mf n$, and opposite radical $\mf n^-$.  We note that the form $(\_|\_)$ on $\h$ induces a non-degenerate bilinear form on $\h^\ast$.

We set $\Delta_\ob$ to be the set of all simple odd roots with respect to the triangular decomposition \eqref{eq::tri0} and $\mf s^{\alpha}$ to be the root space of $\mf s$ corresponding to $\alpha\in \Delta_\ob$.
Recall $\hat{\mf n}=\mf n\otimes \Lambda(\theta)$ and let $c$ be a non-zero complex number. For a given  character $\chi: \hat{\mf n}_\oa \rightarrow \C$, we associate a character  $\zeta_\chi: \mf n_\oa\rightarrow \C$  by declaring that, for any root vector $X$ in $\mf n_\oa$, the value $\zeta_\chi(X)$ is given by:
\begin{align*}  &\zeta_\chi(X):=  \chi(X) + \left\{\begin{array}{ll} \frac{-(\alpha_1|\alpha_2)}{c} \chi(X_1\otimes \theta)\chi(X_2\otimes \theta), \\ \qquad \text{if $X=[X_1,X_2]$ for some $X_i\in \mf s^{\alpha_i}$ with $\alpha_i \in \Delta_\ob$};\\ 	0,  ~~~ \text{otherwise}. 
	\end{array} \right.   \end{align*} Let $\g$-Wmod$^\chi_c$  be the category   of   finitely-generated $\g$-modules $M$ such that $(z-c)M=0$  and $x-\chi(x)$ acts locally nilpotently on $M$, for all $x\in\hat{\mf n}_\oa.$  We will refer to objects in  $\g$-Wmod$^\chi_c$ as {\em Whittaker modules associated with $\chi$ at level $c$}. In addition,  we define the Whittaker category $\mf s$-Wmod$^{\zeta_\chi}$ as the category of finitely-generated $\mf s$-modules on which  $x-\zeta_{\chi}(x)$ acts locally nilpotently, for each $x\in \mf n_\oa$. We  define Whittaker modules over  the subalgebra $\mf c:= \mf s\otimes \theta+\C z$ in a similar fashion. If we set $\eta:=\chi|_{\mf n_\ob\otimes \theta}$, then it turns out that there is a unique simple Whittaker $\mf c$-module $\mf F_c^\eta$   associated with $\eta$ and $c$, which lifts to a simple Whittaker module over $\g$. Since an $\mf s$-module $M$ may be regarded as a $\g$-module by extending the action trivially, we have that $M\otimes \mf F_c^\eta$ is a $\g$-module. The construction of such a functor is inspired by \cite{KT85,Ch93}.  Our first main result is an equivalence of categories between $\mf s\text{-}\text{Wmod}^{\zeta_\chi}$ and $\g$-\text{Wmod}$^{\chi}_c$:   
 
\begin{thm} \label{thm::1} For an arbitrary character $\chi: \hat{\mf n}_\oa \rightarrow \C$, the  tensor functor 
	$$\_\otimes \mf F_c^\eta:  \mf s\text{-}\emph{Wmod}^{\zeta_\chi}\rightarrow \g\emph{-Wmod}^{\chi}_c,~~ M\mapsto M\otimes \mf F_c^\eta,$$ is an equivalence  of categories. In particular, if $\chi$ vanishes on $\mf n_\oa$, then the categories $\g\emph{-Wmod}^{\chi}_c$ and $\g\emph{-Wmod}^{\chi}_{c'}$ are equivalent, for all $c,c'\not=0$.
\end{thm}

\subsubsection{} \label{sect::11:2}

Let $e$ be an odd  nilpotent element inside a subalgebra $\langle F,f,h,e,E \rangle \subseteq  \mf s$ isomorphic to $\mf{osp}(1|2)$ such that $\langle E,h, F\rangle$ consists of an $\mf{sl}(2)$-triple, where $[e,e]=2E$ and $[f,f]=-2F$.   By the $\mf{sl}(2)$  representation theory, we have a $\Z$-grading
$$\g=\bigoplus_{i\in \Z} \g(i),~\text{ where $\g(i) = \{x\in \g|~\ad h(x)=ix\}$}. $$
We define the $\Z$-graded subalgebra   $\mf m:= \bigoplus_{i\leq -1}\mf g(i).$ 
Let $\chi^e \in \mf m^\ast$  be the unique linear map  determined by:
\begin{align*} 
\chi^e(x)  = (e|\frac{\partial x}{\partial\theta}),\quad x\in\mf m. \end{align*} 
Then $\chi^e(\_): \mf m\rightarrow\C$ defines a character of $\mf m$. Let $Q_{\chi^e}$ denote the  generalized Gelfand-Graev module, i.e., $Q_{\chi^e}:=U(\mf g)\otimes_{U(\mf m)} \C_{\chi^e}$. Set $Q^c_{\chi^e}=Q_{\chi^e}/(z-c)Q_{\chi^e}$ and define the {\em finite supersymmetric (SUSY) $W$-algebra of $\mf s$ associated to $e$ at level $c$}: 
\begin{align*}
\mc{SW}_c(\mf s,e):= \End_{U(\g)}(Q^c_{\chi^e})^{\text{op}}.
\end{align*}
Since $\End_{U(\g)}(Q^c_{\chi^e})^{\text{op}}\cong(Q^c_{\chi^e})^{\ad{\mf m}}$, we have $(U^c(\mf g)\otimes_{U(\mf m)} \C_{\chi^e})^{\ad{\mf m}}\cong\End_{U(\g)}(Q^c_{\chi^e})^{\text{op}}$, where we recall that $U^c(\g):=U(\g)/U(\g)(z-c)$.
Such a formulation of $\mc{SW}_c(\mf s,e)$  is reminiscent of the original definition of finite W-algebras for semisimple Lie algebras introduced by Premet in \cite{Pr02}.

Let $\mc{SW}_c(\mf s,e)\mod$ be  the category of finitely-generated $\mc{SW}_c(\mf s,e)$-modules. We establish in Theorem \ref{thm1} and Appendix \ref{sect::append} a crucial Skryabin type equivalence between $\mc{SW}_c(\mf s,e)\mod$ and the category of finitely generated $U^c(\g)$-modules on which $\mf m_{\chi^e}$ acts locally nilpotently.

Suppose now that $e$ is principal nilpotent, i.e., $E:=\frac{1}{2}[e,e]$ is (even) principal nilpotent in $\mf s_\oa$. In this case, there exists a triangular decomposition as in \eqref{eq::tri0} with nilradical $\mf n$ such that $\mf m=\mf n\otimes \Lambda(\theta)$.
Recall the notations $\chi:=\chi^e{|_{\mf m_\oa}}$ and $\zeta_\chi$ introduced in Subsection \ref{sect::111}. Our second main result is an equivalence of categories $\mf s\text{-}\text{Wmod}^{\zeta_\chi}$ and $\mc{SW}_c(\mf s,e)\mod$. Before stating it,  we note that $\g\text{-Wmod}^{\chi}_c$ can be identified as the category of finitely generated $U^c(\g)$-modules on which $x-\chi(x)$ acts locally nilpotently, for all $x\in \hat{\mf n}_\oa { (=\mf m_\oa)}$.  Let us introduce   
 the associated {\em Whittaker functor} 
\begin{align*}
	&\text{Wh}^c_{\chi{^e}}(\_): \g\text{-Wmod}^{\chi}_c \rightarrow \mc{SW}_c(\mf s,e)\mod, \\ 
&M\mapsto \text{Wh}^c_{\chi{^e}}(M)	:=\{v\in M|~xv=\chi{^e}(x)v,~\text{for all }x\in \mf m\}.
\end{align*}Here $\text{Wh}^c_{\chi{^e}}(M)$ 	is naturally an $\mc{SW}_c(\mf s,e)$-module by letting $\ov{y}.v=yv$, for $v\in \text{Wh}^c_{\chi{^e}}(M)$ and $\ov{y}\in Q^c_{\chi^e}$, where $\ov{y}$ denotes the the coset associated to $y\in U(\g)$. Now, the Whittaker functor is the SUSY version of Skryabin equivalence of categories. Combining this with Theorem \ref{thm::1} we arrive at our second main result: 
\begin{thm}\label{thm::2}
 We have an equivalence of categories 
\begin{align*}
	&\emph{Wh}^c_{\chi{^e}} \circ  (\_\otimes \mf F_c^\eta): \mf s\text{-}\emph{Wmod}^{\zeta_\chi}\rightarrow \mc{SW}_c(\mf s,e)\mod. 
\end{align*}
\end{thm}
In particular, \mbox{Theorem \ref{thm::2}} provides a complete classification of irreducible representations of $\mc{SW}_c(\mf s,e)$. Furthermore, combining Theorem \ref{thm::2} with Skryabin equivalence for $W$-superalgebras associated with even nilpotent elements, see, e.g.~\cite[Theorem 4.1]{SX20}, gives an equivalence of the categories between $\mc{SW}_c(\mf s,e)\mod$ and the category of finitely-generated modules over the principal finite $W$-superalgebra $U(\mf s,E)$, associated with the even principal nilpotent element $E$; see also Corollary \ref{Cor::23}.

 In light of \cite{Los12, SX20, CW24}, we believe that Theorem \ref{thm::1} could play a crucial role in the understanding of $\mc{SW}_c(\mf s,e)$-modules in category $\mc O$ as in \cite{BGK08} when $e$ is of standard Levi type, and hence of the representation theory of SUSY $W$-algebras.

\subsubsection{} This paper is organized as follows. Section \ref{sect::2} is devoted to setting up the preliminaries. In Section \ref{sect::repcat}, we focus on the study of representation categories of $\g$. We discuss in  Subsection \ref{sect::catO} the category $\mc O$ and  the  Fock spaces associated to the non-critical levels $c$, which are used in the sequel. In Subsection \ref{sect::WhMod}, we develop the theory of $\g$-Whittaker modules at non-critial levels, including several essential ingredients for our main results. We describe the twisted Fock space $\mf F_c^\eta$ and investigate the  tensor functor $-\otimes\mf F_c^\eta$.  The proof of Theorem \ref{thm::1} is established in Subsection \ref{sect::223}.

  In Section \ref{sect::4}, we  study the finite SUSY $W$-algebras. We provide in Theorem \ref{thm1} the Skryabin type  equivalence mentioned above.   Combining our results, we obtain in Subsection \ref{sect::43} a proof of  the equivalence  stated in Theorem \ref{thm::2}.  Finally, Appendix \ref{sect::append} is devoted to a  proof of Theorem \ref{thm1}. 

\vskip 0.3cm

{\bf Acknowledgment}. The first two authors are partially supported by National Science and Technology Council grants of the R.O.C., and further acknowledge support from the National Center for Theoretical Sciences. The third author is supported by NRF Grant, \#2022R1C1C1008698 and Seoul National University, Creative-Pioneering Researchers Program.
\vskip 0.3cm
\section{Preliminary} \label{sect::2}
Throughout the paper,  the symbols $\Z$, $\Z_{\geq 0}$, and $\Z_{\leq 0}$ stand for the sets of all, nonnegative and non-positive integers, respectively. All vector spaces and algebras  are assumed to be over the field $\C$ of complex numbers.  Set $\Z_2 = \{\oa, \ob\}$ to be the abelian group of order $2$. For a superspace $V=V_\oa\oplus V_\ob$, we let $p(\_): V_\oa \cup V_\ob\rightarrow \Z_2$ denote the parity function, that is, $p(x) =\oa$ if $x\in V_\oa$ and $p(x) =\ob$, if $x\in V_\ob$. Furthermore, we  let $\Pi$  denote the parity reversing functor  on the category of vector superspaces, that is, for a superspace $V=V_\oa\oplus V_\ob$, we define $\Pi(V)$ to be the superspace with $\Pi(V)_i=V_{i+\ob}$, for $i\in \Z_2$.
\subsection{Basic classical Lie superalgebras} 
Throughout,  we are mainly interested in the following {\em basic classical} Lie superalgebras $\mf s$  over $\C$ from Kac's list \cite{K77}:
\vskip0.2cm
\centerline{$\mathfrak{gl}(m|n),\,\,\mathfrak{sl}(m|n),\,\,\mathfrak{psl}(n|n),\,\,\mathfrak{osp}(m|2n),\,\, D(2,1;\alpha),\,\, G(3),\,\, F(4).$}
In particular, $\mf s$ is {\em quasi-reductive}, i.e., $\mf s_\oa$ is a reductive Lie algebra and $\mf s_\ob$ is semisimple under the adjoint action of $\mf s_\oa$.     
	 Recall that we fix a triangular decomposition $\mf s=\mf n^- \oplus \h \oplus \mf n$ in \eqref{eq::tri0} with the Cartan subalgebra $\h$. We then have a root space decomposition 	$\mf s=\bigoplus_{\alpha\in\Phi\cup \{0\}}\mf s^\alpha,~\mbox{with }\;\mf s^\alpha=\{x\in\mf s\,|\, [h,x]=\alpha(h)x,\;\forall h\in\mf h\},$ 	where  $\Phi\subseteq\mf h^\ast$ is the set of roots.
  We shall denote the sets of the corresponding positive and negative roots  by $\Phi^+$ and $\Phi^-$, respectively. The sets of even and odd roots are denoted respectively by $\Phi_{\oa}$ and $\Phi_{\ob}$ with self-explanatory notations $\Phi^\pm_i: =\Phi^\pm\cap \Phi_i$ for $i\in \Z_2$. The simple system in $\Phi^+$ will be denoted  by $\Delta$. Furthermore, we set $\Delta_i: =\Delta\cap \Phi_i$, for $i\in \Z_2$.   Recall that  $(\_|\_)$ denotes an even  non-degenerate  invariant supersymmetric bilinear
form of $\mf s$, which gives rise to non-degenerate bilinear forms of $\mf h$ and $\h^\ast$. We refer to \cite[Theorem 1.18, Remark 1.19]{CW12} for more details.


 \subsection{Central extensions} \label{sect::cen::212}
    Let  $\Lambda(\theta)$ be the Grassmann superalgebra in one indeterminate $\theta$. We form the Lie superalgebra $\mf s\otimes \Lambda(\theta)$ and let $[\_,\_]_0$ be its Lie bracket. Define a supersymmetric bilinear form $(\_|\_)'$ on $\mf s\otimes \Lambda(\theta)$ by   
\begin{align*}
	&(s_1\otimes f_1| s_2\otimes f_2)':= (-1)^{p(f_1) p(s_2)}(s_1|s_2)\int f_1\wedge f_2,
\end{align*}
for homogeneous $s_1,s_2\in \mf s$, and $f_1,f_2\in \Lambda(\theta)$, where $\int f_1\wedge f_2$ is the coefficient of $\theta$ in $f_1 f_2$. Then $(\_|\_)'$ gives rise to  an odd non-degenerate invariant form on $\mf s\otimes \Lambda(\theta)$. 

Let $D$ denote the derivation $\frac{\partial}{\partial\theta}$ of $\mf s\otimes \Lambda(\theta)$.  Following \cite[Section 1]{Ch93}, we can construct the corresponding 2-cocycle \begin{align}
&\alpha_D(x,y):=(D(x)|y)',~\forall x,y\in \mf s\otimes \Lambda(\theta), \label{lem::111}
\end{align} since $D$ satisfies the so-called {\em superskewsymmetric condition} of loc.~cit.: $$(D(x)|y)' =(-1)^{1+p(x)}(x|D(y))'= (-1)^{p(D)\cdot p(x)}(x|-D(y))',$$ for homogenous $x,y\in \mf s\otimes \Lambda(\theta)$. Consequently, we have the corresponding  central extension of $\mf s\otimes \Lambda(\theta)$: 
\begin{align*}
&0\rightarrow \C z \rightarrow \g \rightarrow \mf s\otimes \Lambda(\theta)\rightarrow 0,
\end{align*} where the bracket $[\_,\_]$ of $\g$ is given by $[x,y] = [x,y]_0+ \alpha_D(x,y)z$, that is, it is determined by  
\begin{align*}
&[s_1\otimes 1, s_2\otimes f] =[s_1,s_2]\otimes f,\\
&[s_1\otimes \theta,s_2\otimes \theta]= (-1)^{ p(s_2)} (s_1|s_2) z,  
\end{align*} for all homogenous elements $s_1,s_2\in \mf s$, and $f\in\Lambda(\theta)$. Note that $\g$ is precisely  the Lie superalgebra with the same notation  introduced in {\eqref{eq::defofg:1}}. More generally, central extensions of $\mf s\otimes\Lambda(n)$, have been determined in \cite[Section 1]{Ch93}, where $\Lambda(n)$ stands for the Grassmann superalgebra in $n$ indeterminates.

  In the paper, we write $\g = \mf s\otimes \Lambda(\theta) \oplus \C z$, as vector spaces.  In addition, for any subspace $\mf a\subseteq \mf s$ and element $x\in \mf s$, we  set 
  \begin{equation} \label{eq:bar}
      \ov{\mf a}:=\mf a\otimes \theta \subseteq \ov{\mf s} \quad  \text{ and }  \quad  \ov x:=x\otimes \theta \in \ov{\mf s}.
  \end{equation}

Recall that we fixed a triangular decomposition for $\mf s$ in \eqref{eq::tri0}. Note that the adjoint action of $\mf t:=\mf h+\C z$ on $\g$ is semisimple. We let $\mf{\hat h}$ denote the centralizer subalgebra $\g^{\mf t}$, that is, $\hat \h= \mf h\otimes \Lambda(\theta)+\C z$ (note that  $\mf t=\hat{\mf h}_\oa$).

	 The adjoint action of $\h$ on $\g$ gives rise to  a root space decomposition $\mf g=\bigoplus_{\alpha\in\Phi\cup \{0\}}\g^\alpha$, with $\g^\alpha=\{x\in\g\,|\, [h,x]=\alpha(h)x,\;\forall h\in\mf h\}$.   This  leads to the following triangular decomposition: 
\begin{align}
&\g =\hat{\mf n}\oplus \hat \h\oplus \hat{\mf n}^-, \label{eq::trian}
\end{align} where $
\hat{\mf n}=\mf n\otimes \Lambda(\theta) =\bigoplus_{\alpha\in \Phi^+}\g^\alpha$ and 
$\hat{\mf n}^-=\mf n^-\otimes \Lambda(\theta) =\bigoplus_{\alpha\in \Phi^-}\g^\alpha$.

\section{Representation Theory} \label{sect::repcat}
In this section, we first introduce the category $\mc O$ and the Fock space $\mf F_c$ over $\g$, for non-zero scalars $c$.  Then we study the Whittaker modules for $\g$ and establish \mbox{Theorem \ref{thm::1}}.

For a given  Lie superalgebra $\mf a$, we denote the category of   finitely generated $\mf a$-modules by $\mf a\mod$.  We denote the universal enveloping algebra of $\mf a$ by $U(\mf a)$.  Throughout this section $\mf s$ denotes a basic classical Lie superalgebra.

\subsection{Category $\mc O$} \label{sect::catO}
\subsubsection{} 
We define the category $\mc O$, with respect to the triangular decomposition \eqref{eq::trian}, as the full subcategory of $\g\mod$ consisting of objects $M$ satisfying the following:
\begin{itemize}
	\item $M$ is semisimple over $\mf t$;
	\item $M$ is locally nilpotent over  $U(\hat{\mf n})$. 
\end{itemize}
Let $M\in \mc O$ and $\nu\in \h^\ast$. We let $M^\nu$ denote the weight space of $\nu$ by 
\begin{align*}
&M^\nu := \{m\in M|~h\cdot m = \nu(h)m, ~\text{for all }h\in \h\}.
\end{align*}
Using the same argument as in the classical  BGG category $\mc O$ (see, e.g., \cite[Section 1.1]{Hu08}), it follows that 
\begin{align*}
&\dim M^\nu <\infty,~\text{ for all }\nu\in \mf h^\ast;\\
&\{\nu \in \mf h^\ast|~M^\nu \neq 0\}\subseteq \cup_{\mu \in X} (\mu -\Z_{\geq 0}\Phi^+),~\text{for some finite subset }X\subset \mf h^\ast.
\end{align*} For any module $M\in \mc O$, we define the {\em character} of $M$ to be $$\ch M:=\sum_{\nu\in \h^\ast} \dim M^\nu e^\nu.$$

\subsubsection{Irreducible $\hat \h$-modules} \label{sect::212}
In this subsection, we review the well-known construction of finite-dimensional simple $\hat \h$-modules (see, e.g., \cite[Section 3]{Pe86}, \cite[Subsection 1.5.4]{CW12}).
For $\la \in \mf t^\ast$,   define a symmetric bilinear form $\langle\_,\_\rangle$ on $\hat \h_\ob$ by \[\langle v,w \rangle_\la:= \la([v,w]), \text{ for }v,w\in \hat \h_\ob.\]
We note that   the radical $\text{Rad}_\la$ of  $\langle\_,\_\rangle_\la$ is zero if  $\la(z)\neq 0$, and all of $\hat \h_\ob$ otherwise. If we set $J_\la$ to be the ideal of $U(\hat \h)$ generated by $\text{Rad}_\la$ and $a-\la(a)$ for $a\in \mf t$, then  $U(\hat \h)/J_\la$ is isomorphic to a Clifford superalgebra $\mc{C}_{\la}$. 
 
Let $\hat \h'\subseteq \hat \h_\ob =\ov{\mf h}$ be a maximal isotropic subspace with respect to the bilinear form $\alpha_D|_{\hat \h_\ob}$ of \eqref{lem::111}.  For $\la\in \mf t^\ast$, we let $\C_\la$ be the one-dimensional $\mf t$-module corresponding to $\la$. Suppose that $\la(z)=0$, then $\C_\la$ extends to a one-dimensional $\hat \h$-module, which we denote by $W_\la$. If $\la(z)\neq 0$, we extend  $\C_\la$ to a $\mf t+\hat \h'$-module by letting $\hat \h' \C_\la=0$. In this case, we use the same notation to define  the following  finite-dimensional $\hat \h$-module $$W_\la:=\Ind_{\mf t+\hat\h'}^{\hat \h} \C_\la.$$

 Recall the well-known fact that $\mc C_\la$ is a finite-dimensional simple associative superalgebra, and thus has a unique simple ($\Z_2$-graded) module  $\widetilde W_\la$; see, e.g., \cite[Chapter 3]{CW12}. Furthermore, the dimension of $\widetilde W_\la$ is  $2^{\lfloor \frac{\dim \h+1}{2} \rfloor}$ if $\la(z)\neq 0$, and it is one-dimensional otherwise. 
In the case when $\la(z)\neq 0$, the action of $U(\hat \h)$ descends to the action of the Clifford superalgebra $\mc C_\la$. Since $W_\la$ is indecomposable, it is simple, and hence we may identify $W_\la$ with $\widetilde W_\la$, up to parity.  This implies that $W_\la$ is simple. We have the following (see, e.g, \cite[Lemma 1.42]{CW12}).
\begin{lem}
$\{W_\la,~\Pi W_\la|~\la\in \mf t^\ast\}$ is an exhaustive list of  all finite-dimensional simple $\hat \h$-modules.
\end{lem}  
We remark that, if $\la(z)\neq 0$, then there is an odd automorphism of $W_\la$ if and only if the dimension of $\hat \h_\ob$ (i.e., $\dim \mf h$) is odd.

\subsubsection{Verma and simple modules} 

A $\g$-module $M$  is said to be a {\em highest weight module} of highest weight $\la\in \mf t^\ast$ if $M$ is generated by an irreducible $\hat \h$-submodule $W\subset M$ isomorphic to $W_\la$ such that $\hat{\mf n}W=0$. 
Let $\hat{\mf b}:=\hat{\mf n}+\hat{\mf h}$. 
 
For each $\la\in \mf t^\ast$ let $W_\la$ be as in Section \ref{sect::212} and regard it as a $\hat{\mf b}$-module by letting $\hat{\mf n}W_\la=0$. Define the  {\em Verma module}
\[M(\la):= \Ind_{\hat{\mf b}}^{\g} W_\la.\]
The following lemma shows that modules in $\mc O$ admit finite filtrations with highest weight sections. 
\begin{lem} \label{lem::33} Let $M\in \mc O$. Then $M$ admits a finite filtration $0=M_0\subseteq M_1 \subseteq \cdots \subseteq M_\ell =M$ such that each $M_{i+1}/M_i$ is a highest weight module.
\end{lem}
\begin{proof} 
It follows by the definition of $\mc O$ that $M$ is generated by a finite-dimensional $\hat{\mf b}$-submodule $V\subseteq M$ (recall that $\h$ is purely even). Therefore, $M$ is an epimorphic image of $\Ind_{\hat{\mf b}}^{\g} V$. Now, the lemma follows from the fact that any finite-dimensional irreducible $\hat{\mf b}$-module is an irreducible $\hat{\mf h}$-module on which $\hat{\mf n}$ acts trivially.
\end{proof}

It follows by a standard argument that $M(\la)$ has a unique simple subquotient, which we shall denote by $L(\la)$.  Thanks to Lemma \ref{lem::33}, we have the following classification of simple objects in $\mc O$. 
\begin{lem} The set $$\{L(\la),~\Pi L(\la)|~\la\in \mf t^\ast\},$$ is an exhaustive list of  all   simple objects in $\mc O$. 
\end{lem}   

\subsubsection{Block decomposition} \label{sect::314} We observe that $\mc O$ decomposes into $\mc O =\bigoplus_{c\in \C} \mc O^c$, where $\mc O^c$ is the full subcategory consisting of objects in $\mc O$ on which $z-c$ acts locally nilpotently. For any $\mf s$-module $M$, we can  lift it to a $\g$-module by letting $\mf c:=\ov{\mf s}+\C z$ act trivially on $M$ . This leads to a full and faithful exact functor $\iota_{\mf s}$ from 
the BGG category $\mc O(\mf s)$ of $\mf s$ with respect to the triangular decomposition \eqref{eq::tri0} to $\mc O^0$. In particular, if $M_{\mf s}(\gamma)$ and $L_{\mf s}(\gamma)$ respectively denote the target images of  the Verma and the simple module over $\mf s$ with highest weight $\gamma\in \h^\ast$ under $\iota_{\mf s}$, then $\iota_{\mf s}(M_{\mf s}(\gamma))$ is a quotient of $M(\gamma)$ and $L_{\mf s}(\gamma) =L(\gamma)$ ($\gamma\in \h^\ast$) exhaust all simple objects of $\mc O^0.$ Here, for each $\gamma \in \mf h^\ast$, we  view it as an element of $\mf t^\ast$ by setting $\gamma(z)=0$.

Therefore, we may focus on the blocks $\mc O^c$ with non-zero central charges $c$.  In this case, for each $L(\la)\in \mc O^c$ (i.e.,  $\la(z) = c\neq 0$), we have that $\dim W_\la >1$ and thus $L(\la)\cong \Pi L(\la)$ if and only if $\dim \hat\h_\ob$ is odd. 

From now on, for a given simple module $L$, it will not important for us to make the distinction between $L$ 
and $\Pi L$,  and so we shall simply use $L$ denote either of them.

\subsubsection{Tensor product of simple modules}

Recall that the algebra of endomorphisms of a  ($\Z_2$-graded) simple module over a Lie superalgebra $\mf a$ is either one-dimensional or  spanned  by the identity and an involution interchanging the even and odd parts of the module; see, e.g., \cite[p.609]{K78} or \cite[Proposition 8.2]{Ch95}. We say a simple $\mf a$-module is of {\em type} $\texttt  A$ if its endomorphism algebra is $\C$ and is of {\em type} $\texttt Q$ otherwise.

\begin{lem} \label{lem::333}
  Every simple $\mf s$-module is of type {\em $\texttt A$}.
\end{lem}
\begin{proof}
Suppose that $L$ is a ($\Z_2$-graded) simple module  over $\mf s$ with the associated homomorphism of Lie superalgebras $\phi: U(\mf s)\rightarrow \End_{\C}(L)$. By the super analogue of Duflo's theorem \cite[Main Theorem 2.2]{Mu92}, the kernel of $\phi$ coincides with the annihilator $\Ann_{U(\mf s)}V$ of an irreducible highest weight module $V$ over $\mf s$. Since the highest weight space of $V$ is one-dimensional, it follows that $V$ is of type $\texttt A$. Therefore, $U(\mf s)/\Ann_{U(\mf s)}V$ is isomorphic to a dense subalgebra of $\End_\C(V)$. Since $U(\mf s)/\Ann_{U(\mf s)}V \cong \phi(U(\mf s))$, we may conclude that $L$ is of type $\texttt A$.   
\end{proof}

Recall that $\mf c$ denotes the nilpotent ideal $\mf s\otimes \theta+\C z$ of $\g$, so that we have $\g =\mf s\oplus\mf c$ as a vector space.
\begin{prop} \label{prop::4}  Suppose that  $S$ is a simple $\g$-module such that   the images of the representation maps from $U(\mf g)$ and $U(\mf c)$ in  $\End_\C(S)$ coincide. Then the tensor product $L\otimes S$ is a simple $\g$-module, for any  simple $\mf \g$-module   $L$ such that $\mf c L=0$.  
\end{prop}
\begin{proof} Let $x$ be a homogeneous non-zero element of $L\otimes S$. Set  $N$  to be the cyclic submodule $U(\g)x$. The goal is to show that $N= L\otimes S$. Let $x =\sum_{i=1}^k v_i\otimes w_i$, for some homogeneous vectors $v_i\in L$ and $w_i\in S$ such that  $w_i$'s are linearly independent. First, suppose that $k=1$. Then $x=v_1\otimes w_1$ generates the $L\otimes S$. To see this, for any homogeneous vector $w\in S$, we can choose an element $c\in U(\mf c)$ such that $cw_1 = w$, and so $ cx=v_1\otimes w \in N$. If $v\in L$ is a homogeneous vector, then we choose $s\in U(\mf s)$ such that $sv_1 =v$. Suppose  $s\in \mf s$. We have $s(v_1\otimes w) - (sv_1)\otimes w= (-1)^{p(v_1)p(s)}v_1\otimes sw \in N$ and it follows that $v\otimes w\in N$. The general case can be argued inductively. Therefore we have $N =L\otimes S$ in this case.

Next, we assume that $k$ is an arbitrary positive integer strictly bigger than $2$.  If $S$ is of type $\texttt A$ as a $\mf c$-module, then  there exists an element $r\in U(\mf c)$ such that $rw_1\neq 0$ and $rw_i=0$ for  $i\neq 1.$  In this case, we have $rx =v_1\otimes w_1$ and thus the conclusion follows by the argument above. If $S$ is of type $\texttt Q$ as a $\mf c$-module, then it follows by the Burnside's theorem for superalgebras (see, e.g., \cite[Proposition 8.2]{Ch95}) that  there are element $r\in U(\mf c)$ and an odd automorphism $q: S\rightarrow S$ such that  $$ 0\neq v_1\otimes w'+v_2\otimes q(w') =rx \in N,$$ for some homogeneous  vectors $v_1, v_2\in L$ and $w'\in S$. Since $\mf cL=0$, it follows that  elements of the form $v_1\otimes w+v_2\otimes q(w)$  lie in $N$, for $w\in S$. Therefore, for any $x\in \mf s$ and $w\in S$, we have  $$(xv_1)\otimes w+(xv_2)\otimes q(w) - x\cdot (v_1\otimes w+v_2\otimes q(w))\in N.$$
By Lemma \ref{lem::333} there is an element $r\in U(\mf s)$ such that $rv_1 \neq 0$ and $rv_2=0$, and hence  $N$ contains a non-zero homogenous element of the form $rv_1\otimes w$. Consequently, $N=L\otimes S$. This completes the proof. 
 \end{proof}

\subsubsection{Fock space $\mf F_c$} \label{sect::Fock} In this subsection  $c$ denotes a fixed non-zero complex number. 
Recall the  bilinear form $\alpha_D$ from \eqref{lem::111}. Let $I\subseteq\ov{\mf s}$ be a maximal isotropic subspace with respect to $\alpha_D$ containing $\ov{\mf n}$.   Let $\C |0\rangle$ be the one-dimensional module over $I+\C z$ spanned by the vector $|0\rangle$  determined by $z|0\rangle = c|0\rangle$ and $I |0\rangle =0$. We define the following $\mf c$-module
\begin{align*}
&\mf F_c := \Ind^{\mf c}_{I+\C z}|0\rangle, 
\end{align*} 
which is referred to as {\em Fock space}.
 	Let $\phi_0: U(\mf c)\rightarrow \End_\C(\mf F_c)$ denote the homomorphism associated with the Fock space representation $\mf F_c$. To describe the image of $\phi_0$,  we choose $E_\alpha\in \mf s^\alpha$ and   $F_\alpha\in \mf s^{-\alpha}$ such that $(E_\alpha|F_\alpha) =1$, for each $\alpha \in \Phi^+$. Furthermore, we choose an orthonormal basis $\{H_i\}_{i=1}^\ell$ of $ \mf h$ with respect to $(\_|\_)$. Let $U$ denote the underlying superspace $\ov{\mf n}^-:=\mf n^- \otimes \theta$. Recall that the  {\em Weyl-Clifford superalgebra $\mf{WC}(U)$} associated with $U$ can identified with the superalgebra of polynomial differential operators on the supersymmetric algebra  $\mc S(U ) = \mc S( U_\oa)\otimes \Lambda( U_\ob),$ where $\mc S(U_\oa)$ and  $\Lambda(U_\ob)$ are identified respectively with the polynomial algebra in  variables $\{{\ov{F_\alpha}}|~\alpha\in \Phi^+_\ob\}$ and the Grassmann superalgebra in  indeterminates $\{{\ov{F_\alpha}}|~\alpha\in \Phi^+_\oa\}$; see, e.g., \mbox{\cite[Section 5.1]{CW12}}. For $\alpha\in \Phi^+$, we define the corresponding operators of multiplication $\ell_{\ov{F_\alpha}}$ and   (super-)derivation  $D_{\ov{F_\alpha}}$: 
	\begin{align*}
	&\ell_{\ov{F_\alpha}}:~\mf F_c\rightarrow \mf F_c,~ f\mapsto \ov{F_\alpha} f,\\
	& D_{\ov{F_\alpha}}:~\mf F_c\rightarrow \mf F_c,~ f\mapsto  (-1)^{p(\ov{F_\alpha})}\frac{\partial}{\partial \ov{F_\alpha}}f.
	\end{align*}
 
Realizing in terms of differential operators, we can identify $\phi_0({\ov{F_\alpha}})$ and $\phi_0({\ov{E_\alpha}})$ with $\ell_{\ov{F_\alpha}}$  and $cD_{\ov{F_\alpha}}$, respectively. We refer to \cite[Proposition 5.3]{CW12} for more details. 
	 
Recall the Clifford superalgebra $\mc C_{\la}$  from Subsection \ref{sect::212}, for $\la \in \mf t^\ast$ with $\la(z)\neq 0$.	We can now identify the action of $\mf c$ on  $\mf F_c$ with the action of the tensor product $\mf{WC}(U)\otimes \mc C_{\la}$, that is, $\phi_0(U(\mf c))$ is generated by $\{\phi_0(\ov{F_\alpha}),~\phi_0(\ov{E_\alpha}),~ \phi_0(\ov{H_i})|~\alpha\in \Phi^+,~1\leq i \leq \ell \}$
 subject to the following defining relations:
\begin{equation}
\begin{aligned}
 	&[\phi_0(\ov{E_\alpha}), \phi_0(\ov{E_\beta})]= [\phi_0(\ov{F_\alpha}), \phi_0(\ov{F_\beta})] =  [\phi_0(\ov{E_\alpha}), \phi_0(\ov{H_i})] = [\phi_0(\ov{F_\alpha}), \phi_0(\ov{H_i})] =0, \\
 	&[\phi_0(\ov{E_\alpha}), \phi_0(\ov{F_\beta})] = (-1)^{p(E_\alpha)}c\delta_{\alpha,\beta}  \\
 	&[\phi_0(\ov{H_i}), \phi_0(\ov{H_j})] =c\delta_{i,j}, \label{eq::11}
 \end{aligned}\end{equation} for $\alpha, \beta \in \Phi^+$ and $1\leq i,j\leq \ell$.

The Fock space  $\mf F_c$ is a simple $\mf c$-module. Next, we are going to lift  $\mf F_c$  to  a $\g$-module.  We first define the following  dual bases  of $\mf s$ with respect to the bilinear  form $(\_|\_)$: \begin{equation}
\begin{aligned}
	&\{u^i\}_{i=1}^q := \{E_\alpha, F_\alpha |~\alpha\in \Phi^+\} \cup \{H_1,\ldots, H_\ell\},\\ 
		&\{u_j\}_{j=1}^q := \{F_\alpha, (-1)^{p(E_\alpha)}E_\alpha, |~\alpha\in \Phi^+\}\cup \{H_1,\ldots, H_\ell\}.  \label{eq::5::dual}
\end{aligned}
\end{equation}
Namely, we have $(u^i|u_j) =\delta_{ij}$, for $1\leq i,j\leq q$. Note that  $[s, u_i] =\sum_{j=1}^q(u^j|[s,u_i])u_j$, for any $s\in \mf s$.  We state this in a slightly more general form that we shall need later on in the following lemma.

\begin{lem} \label{lem::6::ext}
	Suppose that $S$ is a simple $\mf c$-module on which $z$ acts as a non-zero scalar  $c\in \C$.   Let $\phi: U(\mf c)\rightarrow \End_\C(S)$ be the corresponding  representation map. Then    $S$ lifts  to a simple $\g$-module by extending $\phi$  to $U(\mf g)$ by defining 
	\begin{align}
		&\phi(s) :=  \frac{1}{2c} \sum_{j=1}^{q}\phi(\ov{[s,u^j]})\phi( \ov{u_j}), \text{ for any }s\in \mf s. \label{eq::lem6::ext}
	\end{align}  
\end{lem}
\begin{proof} 
  The goal is to show that the following identities hold in $\End_\C(S)$:
\begin{align}
	&[\phi(s), \phi(\ov{u_k})] = \phi{[s,\ov{u_k}]},   \label{eq::5} \\
	&[\phi(s), \phi(t)] = \phi([s,t]),  \label{eq::6}
\end{align} for any homogenous elements $s,t\in \mf s$ and $1\leq k\leq q$. For simplicity of notations, we will write $x$ instead of $\phi(x)$ in the following calculations, for any $x\in \mf c$.   For each $1\leq i\leq q$, we set $p_i:=p(u^i)$. We note that $\phi(s)=\frac{1}{2c} \sum_{i,j=1}^{q} (-1)^{p_ip_j+p_j+1} (u^j| [s,u_i]) \ov {u^i}\, \ov{u_j}$ by definition.   To check the equation \eqref{eq::5}, we calculate the element $[\phi(s),\ov{u_k}]$ in $\End_\C(S)$: \begin{align*}
	&[\phi(s),\ov{u_k}]  = \frac{1}{2c}\sum_{i,j=1}^{q} (-1)^{p_ip_j+p_j+1}(u^j|[s,u_i])[\ov{u^i}\,\ov{u_j}, \ov{u_k}]  \\
	& = \frac{1}{2c}\sum_{i,j=1}^{q}(-1)^{p_ip_j+p_j+1}(-1)^{(p_j+1)(p_k+1)}(u^j|[s,u_i])\,[\ov{u^i},\ov{u_k}] \,\ov{u_j}\\ &+\frac{1}{2c}\sum_{i,j=1}^{q} (-1)^{p_ip_j+p_j+1} 
	(u^j|[s,u_i])\,\ov{u^i}\,[\ov{u_j}, \ov{u_k}] \\
	& = \frac{1}{2}\ov{[s,u_k]}+\frac{1}{2}(-1)^{p(s)p_k+1}\ov{[u_k, s]} 
	 = \ov{[s,u_k]}.  \end{align*}

Next, we prove the equation \eqref{eq::6}. First,  it follows by the definition of $\phi$ in \eqref{eq::lem6::ext} that   \begin{align*}
	&\phi([s,t]) = \frac{1}{2c}\sum_{j=1}^{q} \ov{[[s,t],u^j]}\,\ov{u_j}.  \end{align*}
We calculate the element $[\phi(s),\phi(t)]$ in $\End_\C(S)$:
\begin{align*}
	&4c^2 [\phi(s),\phi(t)] = \sum_{j,k=1}^{q} [\ov{[s,u^j]}\, \ov{u_j},\ov{[t,u^k]}\,\ov{u_k}] \\
	&= \sum_{j,k=1}^{q} (-1)^{(p_j+1)(p(t)+p_k+1)+p(s)+p_j} ([s,u^j]|[t,u^k])\ov{u_j}\,\ov{u_k}c \\
	&+ \sum_{j=1}^{q} (-1)^{p(s)p(t)+1} \ov{[t,[s,u^j]]}\, \ov{u_j}c\\
	&+\sum_{j,k=1}^{q} (-1)^{p_j} (u_j|[t,u^k]) \ov{[s,u^j]}\, \ov{u_k} c \\
	&+ \sum_{j,k=1}^{q} (-1)^{(p_j+p(s)+1)(p(t)+p_j+1)+(p(t)+p_k+1)(p_j+1)+p_j} (u_j|u_k) \ov{[t,u^k]}\,\ov{[s,u^j]}c \\
	& = \sum_{k=1}^{q} \ov{[s,[t,u^k]]}\,\ov{u_k}c +\sum_{j=1}^{q} (-1)^{p(s)p(t)+1} \ov{[t,[s,u^j]]}\,\ov{u_j} c \\
	& +\sum_{k=1}^{q}\ov{[s,[t,u^k]]}\ov{u_k}c + \sum_{k=1}^{q} (-1)^{p(s)p(t)+1} \ov{[t,[s,u^k]]}\,\ov{u_k} c \\
	& = 2c \sum_{k=1}^{q} \left(\ov{[s,[t,u^k]]} \,\ov{u_k}+(-1)^{p(s)p(t)+1} \ov{[t,[s,u^k]]}\,\ov{u_k}\right) \\
	&=2c\sum_{k=1}^{q} \ov{[[s,t],u^k]}\, \ov{u_k}.
\end{align*} Therefore we have $[\phi(s),\phi(t)] =\phi([s,t])$. This completes the proof. \end{proof}

 Let $\rho\in \mf h^\ast$ denote the Weyl vector of $\mf s$, that is, \[\rho:= \frac{1}{2}({\sum_{\alpha\in \Phi_\oa^+}\alpha}-\sum_{\beta\in \Phi_\ob^+}\beta).\]  
Let $\rho_c \in \mf t^\ast$ be determined by $\rho_c|_{\mf h} =\rho$ and $\rho_c(z) = c$. We have the following lemma.

\begin{prop} \label{prop::5} 
The	formula \eqref{eq::lem6::ext} extends  the $\mf c$-module  $\mf F_c$ to a simple highest weight $\g$-module of the highest weight $\rho_c$.
\end{prop}
\begin{proof} 
 
 By Lemma \ref{lem::6::ext}, $\mf F_c$ lifts to a simple $\g$-module. We shall show that $H|0\rangle = \rho(H)|0\rangle$, for each $H\in \mf h$. To see this, recall from \eqref{eq::5::dual} that we have specified the dual bases $\{u^i\}_{i=1}^q$ and $\{u_j\}_{j=1}^q$ of $\mf s$ with respect to the bilinear form $(\_|\_)$ and root vectors $E_\alpha\in \mf s^\alpha$ and $F_\alpha \in \mf s^{-\alpha}$, for each positive root $\alpha\in \Phi^+$. 
We calculate
\begin{align*}
H |0\rangle &= \frac{1}{2c} \sum_{j=1}^{q}\ov{[H,u^j]} \,\ov{u_j}  |0\rangle \\
&= \frac{1}{2c} \left(\sum_{\alpha \in \Phi^+} \alpha(H)\ov{E_\alpha}\,\ov{F_\alpha}+\sum_{\alpha \in \Phi^+} (-1)^{1+p(E_\alpha)} \alpha(H)\ov{F_\alpha}\,\ov{E_\alpha}\right)  |0\rangle\\
&= \frac{1}{2c} \sum_{\alpha \in \Phi^+} \alpha(H)\ov{E_\alpha}\,\ov{F_\alpha}   |0\rangle \\
&= \frac{1}{2c} \sum_{\alpha \in \Phi^+} \alpha(H)\left((-1)^{p(E_\alpha)}z+(-1)^{p(E_\alpha)}\ov{F_\alpha}\,\ov{E_\alpha}\right)  |0\rangle \\
&= \rho(H)   |0\rangle 
\end{align*}  Also, by a direct calculation, we find that $E_\beta|0\rangle =0$, for any $\beta\in \Phi^+$.  This shows that  $\mf F_c\cong L(\rho_c)$. 
\end{proof}

Recall from Subsection \ref{sect::314} that $L_{\mf s}(\mu)$ and $M_{\mf s}(\mu)$ respectively denote  simple and Verma modules in $\mc O^0$, lifted from the corresponding highest weight $\mf s$-modules, for $\mu \in \h^\ast$. The following is an immediate consequence of Propositions \ref{prop::4} and \ref{prop::5}.  
\begin{cor} \label{cor::11}
Let $\la \in \mf t^\ast$ be such that $\la(z)=c\neq 0$. Then we have 
\begin{align*}
&L(\la)\cong \mf F_c\otimes L_{\mf s}(\la|_{\mf h} -\rho).
\end{align*}
\end{cor}
\begin{rem} 
	Since the dimension of $W_\la$ is $2^{\lfloor \frac{\dim \h+1}{2} \rfloor}$, we have the following character formula 
	\begin{align}
	&\ch L(\la) =2^{\lfloor \frac{\dim  \h+1}{2} \rfloor} \frac{\prod_{\alpha\in \Phi_\ob^+} (1+e^{-\beta})}{\prod_{\beta\in \Phi_\oa^+} (1-e^{-\alpha})} \ch L_s(\la|_{\mf h}-\rho). \label{eq::ch7}
	\end{align}
In particular, in the case when $\mf s$ is of type $\mf{gl}$ or $\mf{osp}$, the formula in  \eqref{eq::ch7} provides a complete solution to the irreducible character problem of $\g$ in terms of the Kazhdan–Lusztig combinatorics for basic classical Lie superalgebras $\mf s$ 
established in  \cite{Br03, Ba17, BW18, CLW11, CLW15}.
\end{rem}

For each $\gamma \in \mf h^\ast$, we   view it as an element of $\mf t^\ast$ by setting $\gamma(z)=0$. The following corollary gives a description of the composition multiplicities of  Verma modules. 
\begin{cor} Suppose that $c\neq 0$. Then the full subcategory $\mc O^c$ is artinian.  Furthermore, for any $\la\in \mf t^\ast$ be such that $\la(z)=c$, we have \begin{align} &\ch M(\la) 
		 =   \sum_{\mu\in \mf t^\ast}[M_{\mf s}(\la|_\h-\rho): L_{\mf s}(\mu|_{\h}-\rho)] \ch L(\mu). \label{eq::8} \end{align}   
\end{cor}
\begin{proof}
	Let $\la(z)=c \neq 0$. 
	We calculate 
 \begin{align*}
 \ch M(\la) &=\ch \mf F_c\cdot \ch M_{\mf s}(\la|_{\mf h}-\rho)    = \ch \mf F_c \cdot  \sum_{\nu\in \mf h^\ast}[M_{\mf s}(\la|_\h-\rho): L_{\mf s}(\nu)] \ch L_{\mf s}(\nu)\\
 &=\sum_{\nu\in \mf h^\ast}[M_{\mf s}(\la|_\h-\rho): L_{\mf s}(\nu)] \ch L(\nu+\rho_c).
 \end{align*}  Thanks to Corollary \ref{cor::11}, we have established the formula \eqref{eq::8}, which implies that $M(\la)$ has finite length. The conclusion that $\mc O^c$ is artinian follows by Lemma \ref{lem::33}. \end{proof}

\subsection{Whittaker modules} \label{sect::WhMod}
\subsubsection{Whittaker categories} \label{sect::221} Let $\g$-Wmod denote the category of  finitely-generated $\g$-modules $M$ such that $M$ is locally $U(\hat{\mf n})$-finite and the central element $z$ acts on $M$ as a non-zero scalar multiple of the identity. We call modules in $\g$-Wmod {\em Whittaker modules (at non-critical levels)} by analogy with the Whittaker modules of classical Lie (super)algebras (see, e.g., \cite{BCW14,C21}).

 To explain  further, we need several preparatory results. Since the bilinear form $(\_|\_)$ is invariant, it follows that if  $x\in \mf s_\oa$  such that its adjoint action $\ad x:\mf s_\oa\rightarrow\mf s_\oa$ is nilpotent then  the inner automorphism $\text{exp}(\ad x) = \sum_{k=0}^\infty (\ad x)^k/k!$ of $\mf s_\oa$ extends to an automorphism of $\g$.   Let $G$ be the subgroup of $\text{Aut}(\g)$ generated by all such inner automorphisms.  Since Borel subalgebras of reductive Lie algebras are conjugate under $G$, we conclude that $\g$-Wmod is independent of the choice of the triangular decomposition given in \eqref{eq::tri0}, up to equivalence.  

Let $c\in \C$ be a non-zero complex number. For any character  $\chi: \hat{\mf n}_\oa\rightarrow \C$, let $\g$-Wmod$^\chi_c$ denote the full subcategory of $\g$-Wmod consisting of objects $M$ such that $(z-c)M=0$ and $x-\chi(x)$ acts locally nilpotently on $M$, for all $x\in\hat{\mf n}_\oa.$  When the scalar $c$ is clear from the context, we say a $\g$-module is a {\em Whittaker module associated to $\chi$} if it lies in $\g$-Wmod$^\chi_c$.  
 Since the adjoint action of $\hat{\mf n}_\oa$ on $\g$ is nilpotent, we observe that $\g$-Wmod decomposes  into a direct sum  $\g$-Wmod$= \bigoplus_{\chi,c}\g$-Wmod$^\chi_c$, where $\chi$ and $c$ run over the characters of $\hat{\mf n}_\oa$ and non-zero complex numbers, respectively.  
 
  Similarly, we denote by $\mf s$-Wmod the category of finitely-generated $\mf s$-modules that are locally finite over $U(\mf n)$.  For an arbitrary character $\zeta$ of $\mf n_\oa$, we let  $\mf s$-Wmod$^{\zeta}$ be the full subcategory of $\mf s$-Wmod consisting of objects on which $x-\zeta(x)$ acts locally nilpotently, for any $x\in \mf n_{\bar{0}}$.

\subsubsection{Locally $U(\ov{\mf n})$-finite simple $\mf c$-modules} Recall the representation map $\phi_0: U(\mf c)\rightarrow \End_\C(\mf F_c)$ from Subsection \ref{sect::Fock}. 
For any character  $\eta: \ov{\mf n}_\oa\rightarrow \C$, we may extend it to a character of $\ov{\mf n}$ by letting $\eta(\ov{\mf n}_\ob)=0$. We can define an even linear map  $\phi_\eta: \mf c\rightarrow \End_\C(\mf F_c)$  by   
  \begin{align*}\phi_\eta(x):= \left\{\begin{array}{ll}	\phi_0(x) +\eta(x), &  \text{~if $x\in \ov{\mf n}$};\\ 		\phi_0(x), & \text{~if $x\in \ov{\mf n^-}+\ov{\mf h}+\C z$}. 
  	\end{array} \right. \end{align*}
Since elements in  $\{\phi_\eta(\ov{F_\alpha}),~\phi_\eta(\ov{E_\alpha}),~ \phi_\eta(\ov{H_i})|~\alpha\in \Phi^+,~1\leq i \leq \ell \}$ satisfy the defining relations \eqref{eq::11}, $\phi_\eta$ induces a homomorphism $\phi_\eta: U(\mf c)\rightarrow \End_\C(\mf F_c)$. We shall denote this new $\mf c$-module by  $\mf F_c^{\eta}$, which has the same underlying space as $\mf F_c$. The following lemma will be useful.
\begin{lem} \label{lem::11}
 Suppose that $V$ is a simple $\mf c$-module on which $z$ as the non-zero scalar $c\in \C$ and  $x-\eta(x)$ acts nilpotently, for any $x\in \ov{\mf n}$. Then $V\cong \mf F_c^\eta$.
\end{lem}
\begin{proof}
Define the subalgebra $\mf d :=\ov{\mf n}_\oa+\mf c_\ob+\C z $ of $\mf c$. Suppose that $W$ is a finite-dimensional simple $\mf d$-module  such that $(z-c) W=0$. By a straightforward calculation, it follows  that the induced module $\Ind_{\mf d}^{\mf c}W$ is a simple $\mf c$-module isomorphic to $\mf F_c^\eta$.  In particular, if we let $v\in V$ be a non-zero vector such that $xv =\eta(x)v$, for any $x\in \ov{\mf n}$, and let $W$ be a simple submodule of the $\mf d$-module $U(\mf d)v=U(\mf c_\ob)v$, then there is an isomorphism from $\Ind_{\mf d}^{\mf c} W$ to $V$ by adjunction. The conclusion follows.
\end{proof}

\subsection{Whittaker modules induced by $\mf F_c$} \label{sect::223} From now on, we fix  a non-zero complex number $c\neq 0$. 
Recall  the root vectors $E_\alpha, F_\alpha ~(\alpha\in \Phi^+)$ and dual bases $\{u^i\}_{i=1}^q$ and $\{u_j\}_{j=1}^q$ with respect to the form $(\_|\_)$ of $\mf s$ from \eqref{eq::5::dual}. For any character $\eta: \ov{\mf n}_\oa\rightarrow\C$, recall the representation map $\phi_\eta: U(\mf c)\rightarrow \End(\mf F_c^\eta)$. We can lift $\mf  F_c^{\eta}$ to an irreducible $\g$-module via the following action of $\mf s$ by Lemma \ref{lem::6::ext}:
\begin{align*}
&\phi_\eta(s):=\frac{1}{2c} \sum_{j=1}^q \phi_\eta(\ov{[s, u^j]}) \phi_\eta(\ov{u_j}), \text{ for  } s\in \mf s.
\end{align*} For any $\alpha\in \Phi$, we set $\mf n^\alpha:=\mf s^{\alpha}\cap \mf n$.  The following lemma shows that  $\mf F_c^\eta$ is  a   Whittaker module. 

\begin{lem}\label{prop::10} Let $\eta: \ov{\mf n}_\oa\rightarrow \C$  be a character such that $\eta({\ov{\mf n^{\alpha}}})=0$, for any $\alpha\in \Phi^+_\ob\backslash \Delta_\ob$. Let $\hat \eta$ be the character of  $\hat{\mf n}_\oa$ determined by $\hat \eta|_{\ov{\mf n}_\oa} =\eta$ and  the following formula 
\begin{align*}\hat \eta(X):= \left\{\begin{array}{ll} \frac{(\alpha_1|\alpha_2)}{c}\eta(\ov{X_1})\eta(\ov{X_2}), &  \text{if  $X=[X_1,X_2]$ for some $X_i\in \mf n^{\alpha_i}$ with $\alpha_i\in \Delta_\ob$};\\ 	0, & \text{otherwise}, 
		\end{array} \right. \end{align*} 
for any root vector $X\in \mf n_\oa$.	Then we have    $\mf F_c^\eta\in \g$-\emph{Wmod}$_c^{\hat \eta}$. 
\end{lem}
\begin{proof} Suppose that $\beta\in \Phi^+_\oa$, and let  $H_\beta:= [E_\beta, F_\beta]\in \h$ be the co-root of $\beta$. 
	We then calculate  
\begin{align*}
	2c\phi_\eta(E_\beta) |0\rangle  &= \sum_{\alpha\in \Phi^+}\left( \ov{[E_\beta,E_\alpha]}\,\ov{F_\alpha} + (-1)^{p(E_\alpha)}[\ov{E_\beta, F_\alpha]}\,\ov{E_\alpha}\right) |0\rangle +\sum_{i=1}^\ell \ov{[E_\beta, H_i]} \,\ov{H_i} |0\rangle \\
	& = \sum_{\alpha\in \Phi^+} \ov{F_\alpha}\, \ov{X_{\alpha+\beta}} |0\rangle + \sum_{\alpha\in \Phi^+}  (-1)^{p(E_\alpha)}\ov{[E_\beta, F_\alpha]}\,\ov{E_\alpha} |0\rangle  +\sum_{i=1}^\ell \ov{[E_\beta, H_i]}\, \ov{H_i} |0\rangle,\\ 
\end{align*}
for some $X_{\alpha+\beta}\in \mf s^{\alpha+\beta}$. Thus, we have
 \begin{align*}
 2c\phi_\eta(E_\beta) |0\rangle  &
	=   \sum_{\alpha \in \Delta_\ob}  (-1)^{p(E_\alpha)}\ov{[E_\beta, F_\alpha]}\,\ov{E_\alpha}  |0\rangle   -\sum_{i=1}^\ell (H_\beta| H_i) \ov{E_\beta}\, \ov{H_i}  |0\rangle, \\
	& =  - \sum_{\alpha \in \Delta_\ob}  \eta(\ov{[E_\beta, F_\alpha]})\eta(\ov{E_\alpha})  |0\rangle. \end{align*}  
 Therefore, $\phi_\eta(E_\beta)$ vanishes in the case when $\beta$ is not a sum of simple odd roots.  To prove the conclusion, assume that $\beta =\alpha_1+\alpha_2$, for some $\alpha_1,\alpha_2\in \Delta_\ob$. Set $X=[E_{\alpha_1},E_{\alpha_2}]$ to be a non-zero root vector. First, suppose that $\alpha_1$ and $\alpha_2$ are distinct. We calculate \begin{align*}
 &[X,F_{\alpha_1}] = [[E_{\alpha_1},E_{\alpha_2}], F_{\alpha_1}]=  -[H_{\alpha_1},E_{\alpha_2}]  = -(\alpha_2|\alpha_1) E_{\alpha_2}.
 \end{align*} Similarly, we have $[X,F_{\alpha_1}] = -(\alpha_1|\alpha_2) E_{\alpha_1}$. This implies that $$\phi_\eta(X)|0\rangle = \frac{(\alpha_1|\alpha_2)}{c}\eta(\ov{E_{\alpha_1}})\eta(\ov{E_{\alpha_2}})|0\rangle.$$
Next, we assume that $\alpha_1=\alpha_2$, which we denote by $\alpha$. Then we have $$[X,F_\alpha] = -2[H_\alpha, E_\alpha] =- 2(\alpha|\alpha) \eta(\ov{E_\alpha}).$$ Therefore, we obtain that  
\begin{align*}
&\phi_\eta(X)|0\rangle = \frac{(\alpha|\alpha)}{c} \eta(\ov{E_\alpha})^2|0\rangle.
\end{align*}
 This completes the proof.  
\end{proof}  
\begin{rem} \label{rem::13}
Let $\zeta$ be an arbitrary character of $\mf n_\oa$. We can view $\zeta$ as a character of $\hat{\mf n}_\oa$ by letting $\zeta(\ov{\mf n}_\oa)=0$. Suppose that $L$ is a simple Whittaker module over $\g$ associated to $\zeta$  such that  $\mf c L=0$.   When combined with Proposition \ref{prop::4},  Lemma \ref{prop::10} has the consequence that  $L\otimes \mf F_c^{\eta}$ is a simple  object in $\g$-Wmod$^{\zeta+\hat \eta}$, for any character $\eta: \ov{\mf n}_\oa\rightarrow \C$ such that $\eta({\ov{[\mf n_\oa,\mf n_\ob]}})=0$. 
\end{rem}

For a given character $\chi: \hat{\mf n}_\oa\rightarrow\C$, recall from Subsection \ref{sect::111} that the associated character $\zeta_\chi:\mf n_\oa\rightarrow \C$ is defined by  
\begin{align}  &\zeta_\chi(X):=  \chi(X) + \left\{\begin{array}{ll} \frac{-(\alpha_1|\alpha_2)}{c} \chi(\ov{X_1})\chi(\ov{X_2}), \\ \text{\qquad if $X=[X_1,X_2]$ for some $X_i\in \mf s^{\alpha_i}$ with $\alpha_i \in \Delta_\ob$};\\ 	0,  ~~~ \text{otherwise}. 
	\end{array} \right.  \label{eq::14}  \end{align} for any root vector $X$ in $\mf n_\oa$. 	If we let $\eta:= \chi|_{\ov{\mf n}_\oa}$ and $\hat \eta$  be the character of $\hat{\mf n}_\oa$ introduced in Lemma \ref{prop::10}, then $\zeta_\chi$ is exactly the character $\chi|_{\mf n_\oa} -\hat \eta|_{\mf n_\oa}.$ The following proposition will be an important ingredient in the proof of Theorem \ref{thm::1}.
\begin{prop}  \label{prop::13}   Let $\chi: \hat{\mf n}_\oa\rightarrow\C$ be a character, $\eta:= \chi|_{\ov{\mf n}_\oa}$, and $\zeta:=\zeta_\chi$.   Suppose that  $M\in \g\emph{-Wmod}^\chi_c$. Then we have  $M\cong N\otimes \mf F_c^\eta$, for some $N\in \mf s$-\emph{Wmod}$^{\zeta}$.   
\end{prop}
\begin{proof} 
	Let $m\in M$ be a non-zero vector such that $xm = \chi(x)m$, for any $x\in \hat{\mf n}_\oa$. Then the cyclic submodule $V:=U(\mf c)m$ generated by $m$ is isomorphic to the $\mf c$-module $\mf F_c^\eta$. Set $U(\mf s)_k$ to be subspace of $U(\mf s)$ spanned by elements of the form $x_1\cdots x_i$, for some $x_1,\ldots,x_i\in \mf s$ and $1\leq i\leq k$.  We may note that the $\mf c$-submodule $U(\mf s)_k V$ has a (finite) composition series with  composition factors all isomorphic to $\mf F_c^\eta$, and so it is semisimple. Since any self extension of $\mf F_c^\eta$ splits, it follows that  $M$ is a direct sum of simple $\mf c$-submodules, which are all isomorphic to $\mf F_c^\eta$, as  $\mf c$-modules.

	Recall that the $\mf c$-module $V\cong \mf F_c^\eta$ extends to a simple $\g$-module in Subsection \ref{sect::223}. We denote this resulting $\g$-module   by $V_\g$. Since both  $V_\g$ and $M$ are $\g$-modules, we can define the  canonical $\g$-module structure on $\Hom_\C(V_\g, M)$ determined by 
	$x\cdot f(v) = xf(v) -(-1)^{p(x)p(f)} f(xv)$, for any homogeneous elements $f\in \Hom_{\C}(V_\g,M)$, $x\in \mf g$ and $v\in V_\g$. Since $\mf c$ is an ideal of $\mf g$, it follows that the subspace $\Hom_{\mf c}(V_\g, M) = \Hom_\C(V_\g, M)^{\mf c}$ of  $\mf c$-invariants in $\Hom_\C(V_\g, M)$ is a $\g$-submodule, which we will denote by $N$.  	
	 This, together with the fact that $M$ is a completely reducible   $\mf c$-module,  leads to the natural $\g$-module isomorphism   $$N\otimes V_\g \xrightarrow{\cong} M,$$
	sending $f\otimes v$ to $f(v)$, for $f\in N$ and $v\in V_\g.$ 
	
	Since $\mf cN=0$, we can view $N$ as a $\mf s$-module.  It remains to show that  $N$ is an object in $\mf s$-Wmod$^{\zeta}$. To prove that conclusion, we may first assume that $M$ is indecomposable. It follows that  $N$ is indecomposable as well. Let $w\in V$ be a non-zero vector which spans a one-dimensional $\hat{\mf n}_\oa$-module, that is, $U(\hat{\mf n}_\oa)w = \C w$. We may observe that,  for any $v\in  N$, we have   $(xv)\otimes w=(x-\chi(x))\cdot(v\otimes w)$, for all $x\in \hat{\mf n}_\oa$. This implies that   $$(U(\hat{\mf n}_\oa)v)\otimes w \subseteq     U(\hat{\mf n}_\oa)\cdot (v\otimes w),$$  which is finite-dimensional by definition.  This implies that $N$ is locally finite over $U(\mf n)$. Furthermore, we note that $N$ is noetherian since $M\cong N\otimes \mf F_c^\eta$ and $M$ is noetherian. This implies that $N$ is finitely generated.   Now, since $N$ is indecomposable, there is a character $\zeta'$ of $\mf n_\oa$ such that $x-\zeta'(x)$ acts locally nilpotently on $N$, for any $x\in \mf n_\oa$. If we  extend $\zeta'$ to be a character of $\hat{\mf n}_\oa$ by letting $\zeta'(\ov{\mf n}_\oa)=0$, then it follows that $N\in \mf \g$-Wmod$^{\zeta'}_c$. Consequently, we have $M\in \g$-Wmod$^{\zeta'+\hat \eta}_c$ and thus $\zeta' = \chi|_{\mf n_\oa} -\hat \eta|_{\mf n_\oa} =\zeta$, where $\hat\eta$ is defined as above. Therefore, we have $N\in \mf s$-Wmod$^{\zeta}$. 
	
    Now if $M$ is decomposable, then $M$ decomposes into a finite sum of indecomposable modules, since it is finitely generated. This completes the proof. 
\end{proof}

The following is Theorem \ref{thm::1} which we restate for the convenience of the reader.
\begin{thm} \label{thm::1::restt}
	 Let $\chi: \hat{\mf n}_\oa\rightarrow \C$ be a character,  $\eta:= \chi|_{\ov{\mf n}_\oa}$, and $\zeta:=\zeta_\chi$ in \eqref{eq::14}. Then the tensor functor 
	\begin{align*}
	&\_\otimes \mf F_c^\eta: \mf s\emph{-Wmod}^{\zeta} \rightarrow \mf g\emph{-Wmod}^\chi_c
	\end{align*} is an equivalence of categories. 
\end{thm}
\begin{proof}
 	Let  $M\in\g$-Wmod$^{\chi}_c$. We have already seen  in the proof of Proposition \ref{prop::13}  that  $\Hom_{\mf c}(\mf F_c^\eta, M)$ is an object in $\mf s$-Wmod$^\zeta$ such that $M\cong {\Hom_{\mf c}(\mf F_c^\eta, M)}\otimes \mf F_c^\eta$. This gives rise to a  well-defined functor 
\begin{align}
&G(\_):=\Hom_{\mf c}(\mf F_c^\eta,\_): \g\text{-Wmod}^{\chi}_c \rightarrow \mf s\text{-Wmod}^\zeta.
\end{align}  
On the other hand,  if  ${N}\in \mf s$-Wmod$^{\zeta}$, then we have an isomorphism of $\mf s$-modules:     \begin{align}
&{N} \xrightarrow{\cong} \Hom_{\mf c}(\mf F_c^\eta,{N} \otimes \mf F_c^\eta),~v\mapsto \phi_v, \text{ for all }v\in {N},   \label{eq::16}
\end{align} where   $\phi_v$ sends $x$ to $v\otimes x$, for any $x\in \mf F_c^\eta$.   Combined with the fact that any $\g$-module subquotients of  ${N}\otimes \mf F_c^\eta$ are semisimple $\mf c$-modules, it follows  that $N\otimes \mf F_c^\eta$ is Noetherian. Consequently, we have   $N\otimes \mf F_c^\eta\in \mf g\text{-Wmod}^{\chi}_c$ and hence   
 $$F(\_):=\_\otimes \mf F_c^\eta: \mf s\text{-Wmod}^{\zeta} \rightarrow \mf g\text{-Wmod}^\chi_c$$ is a well-defined functor. Now, we have  an adjoint pair $(F,G)$ of functors between  $\mf s\text{-Wmod}^\zeta$ and $\g\text{-Wmod}^{\chi}_c$ with the unit $\varepsilon$ and counit $\varepsilon'$ given respectively by: 
\begin{align}
	& N \xrightarrow{\varepsilon_N} GF(N),~v\mapsto \phi_v, \text{ for all }v\in N, \label{eq::17} \\
	&FG(M) \xrightarrow{\varepsilon'_M} M,~f\otimes v\mapsto f(v), \text{ for all }f \in \Hom_{\mf c}(\mf F_c^\eta,  M), v\in \mf F_c^\eta,\label{eq::18}
\end{align} for any $N\in \mf s\text{-Wmod}^{\zeta}$ and  $ M\in \g\text{-Wmod}^{\chi}_c$. We note that  $\varepsilon_{N}$ is precisely the homomorphism in \eqref{eq::16}, and thus, is an isomorphism. It follows by  Proposition \ref{prop::13} that $\varepsilon'_M$ is an isomorphism. This completes the proof.  \end{proof}

\subsection{Example: $\mf s=\mf{gl}(m|n)$} 
In this subsection, we classify the simple Whittaker modules over $\mf g$ in the case when $\mf s=\gl(m|n)$.  We refer to \cite[Subsection 3.2]{C21} for the details.  Recall that the general linear Lie superalgebra $\mathfrak{gl}(m|n)$ 
can be realized as the space of $(m+n) \times (m+n)$ complex matrices
\begin{align*}
	\mf s:=\gl(m|n):=  	\left\{  \left( \begin{array}{cc} A & B\\
		C & D\\
	\end{array} \right) \|~ A\in \C^{m\times m}, B\in \C^{m\times n}, C\in \C^{n\times m}, D\in \C^{n\times n}
	\right\},
\end{align*} 
with the Lie bracket    given by the super commutator.  Then, $\mf s$ admits a $\Z_2$-compatible $\Z$-grading $\mf s = \mf s_{1}\oplus \mf s_0 \oplus \mf s_{-1}$ given by 
\begin{align*}
	&\mf s_0=\mf s_\oa \cong \gl(m)\oplus \gl(n), \\
	&\mf s_1:= 
	\left\{\begin{pmatrix}
		0 & B \\
		0 & 0
	\end{pmatrix}\|~B\in \C^{m\times n}\right\}\quad\mbox{and}\quad \mf s_{-1}:=
	\left\{\begin{pmatrix}
		0 & 0 \\
		C & 0
	\end{pmatrix}\|~C\in \C^{n\times m}\right\}.
\end{align*}

The bilinear form $(\_|\_)$  can be just chosen to  be the  standard super-trace form.   As we have already explained in Subsection \ref{sect::221}, the Whittaker category $\g$-Wmod is independent of the triangular decomposition. Hence we may fix the Cartan subalgebra $\h$ consisting of all diagonal matrices and the radical subalgebras $\mf n$ and $\mf n^-$ consisting of all strictly upper and strictly lower triangular matrices, respectively. 

Let $\zeta$ be a character of $\mf n_\oa$. Set  $\mf l_\zeta$ to be the subalgebra of $\mf s_\oa$ generated  by $\mf h$ and  $\mf s_\oa^{\pm\alpha}$, for $\alpha \in \Phi^+_\oa$ with $\zeta(\mf s_\oa^\alpha)\neq 0$. Observe that $\mf l_\zeta$ is a Levi subalgebra in the parabolic subalgebra $\mf p = \mf l_\zeta+\mf s_1$. Let $W(\mf l_\zeta)$ denote the Weyl group of $\mf l_\zeta$, which acts on $\h^\ast$ via the dot-action. Let $\h^\ast/\sim$ denote a set of representatives of orbits in $\h^\ast$ under the dot-action of $W(\mf l_\zeta)$.

Building on the
work of McDowell \cite{Mc85} and  Mili{\v{c}}i{\'c} and Soergel \cite{ MS97},  we define   the {\em standard Whittaker modules} (see also \cite{C21}):
\begin{align}
	&M(\la, \zeta):= U(\mf s)\otimes_{U(\mf p)} Y_\zeta(\la, \zeta),~\text{ for } \la \in \h^\ast.\label{def::stdW}
\end{align}
Here $Y_\zeta(\la, \zeta):=U(\mf l_\zeta)/(\text{Ker}\chi_\la^{\mf l_\zeta}) U(\mf l_\zeta)\otimes_{U({\mf n}\cap \mf l_\zeta)}\mathbb C_\zeta$ denotes Kostant's simple Whittaker $\mf l_\zeta$-modules from \cite{Ko78}, where $\chi_\la^{\mf l_\zeta}$ is the central character of $\mf l_\zeta$ associated to $\la$ and $\C_\zeta$ is the one-dimensional ${\mf n}\cap \mf l_\zeta$-module induced by the character $\zeta|_{{\mf n}\cap \mf l_\zeta}$.  It turns out that each $M(\la,\zeta)$ has a simple top, which is denoted by $L(\la,\zeta)$. Furthermore, $\{L(\la,\zeta)|~\la\in \h^\ast/\sim\}$ is the complete list of simple objects in $\mf s$-Wmod$^{\zeta}$.

The following is an immediate consequence of Theorem \ref{thm::1::restt}.
\begin{cor}  Let $\chi: \hat{\mf n}_\oa\rightarrow \C$ be a character,  $\eta:= \chi|_{\ov{\mf n}_\oa}$, and $\zeta:=\zeta_\chi$ in \eqref{eq::14}.  Then the following set 
	\begin{align}
		&\{L(\la,\zeta)\otimes \mf F_c^\eta|~\la\in \h^\ast/\sim\},
	\end{align} 
	is an exhaustive list of pairwise non-isomorphic simple objects in $\g$-\emph{Wmod}$^{\chi}_c$. 
\end{cor} 

\section{Finite SUSY $W$-algebras}  \label{sect::4}
The main purpose of this section is to prove Theorem \ref{thm::2}. We first establish an analogue of the Skryabin  equivalence \cite{Skr} for a class of finite-dimensional graded Lie superalgebras equipped with a ``nilcharacter''. We then study finite SUSY $W$-algebras at non-critical levels and prove the SUSY version of Skryabin equivalence.
\subsection{A Skryabin type equivalence} \label{sect::41}
Throughout this section, we fix a   finite-dimensional   Lie superalgebra $\mf a$ equipped with a $\Z$-grading $\mf a=\bigoplus_{i\in \Z}\mf a(i)$.  
Consider the $\Z$-graded subalgebra  $\mf m= \bigoplus_{i\leq -1} \mf a(i)$ and set $\mf m(i):=\mf m \cap \mf a(i)$, for all $i\in \Z$.    For any $x\in \mf a(i)$, we  denote the degree of $x$ by  $\deg(x):=i$.

Suppose that $\varphi:\mf m\rightarrow \C$ is a character, and let $\C_\varphi$ denote the one-dimensional $\mf m$-module induced by $\varphi$. We consider the following induced $\mf a$-module $$Q_\varphi:=U(\mf a)\otimes_{U(\mf m)} \C_\varphi\cong U(\mf a)/I_\varphi,$$
where $I_\varphi$ is the left ideal of $U(\mf a)$ generated by $x-\varphi(x)$ for $x\in \mf m$, and call $Q_\varphi$ the generalized Gelfand-Graev module associated with $\varphi$.  Let $\pr(\_): U(\mf a)\rightarrow U(\mf a)/I_\varphi$ denote the natural projection. We define an associate superalgebra $\mc W_\varphi$ in the spirit of \cite{Pr02}:  
\begin{align*}
&\mc W_\varphi:=\{ \pr(y)\in U(\mf a)/I_\varphi|~(x-\chi(x))y\in I_\varphi,~\forall x\in \mf m\}, 
\end{align*}
with the multiplication  given by 
$\pr(y_1)\pr(y_2) = \pr(y_1y_2)$, 
 for any $y_1,y_2\in U(\mf a)$ such that $[x,y_1], [x,y_2]\in I_\varphi$, for all $x\in \mf m$. As usual, the algebra $\mc W_\varphi$ can also be identified with the opposite algebra of the endomorphism algebra $\End_{U(\mf a)}(Q_\varphi)$, so that $Q_\varphi$ is a $(U(\mf a),\mc W_\varphi)$-bimodule.

Let $\mf a$-WMod$^\varphi$  denote the category of all  $\mf a$-modules on which $x-\varphi(x)$ acts locally nilpotently, for all $x\in\mf m.$ Also, let $\mc W_\varphi\Mod$ denote the category of all $\mc W_\varphi$-modules.  Define the following left exact functor $\text{Wh}_\varphi(\_)$ from $\mf a\text{-WMod}^\varphi$ to $\mc W_\varphi\Mod$ by  \begin{align}\label{eq:Whitt:func}
	&\text{Wh}_\varphi(M):=\{v\in M|~xv=\varphi(x)v,~\text{for all }x\in \mf m\}, \text{ for $M\in \mf a$-WMod$^\varphi$.}
\end{align} We shall refer it to as the associated {\em Whittaker functor}.

Since $\mf m$ is negatively graded, it follows that the action of $\text{ad}x$ on $U(\mf a)$ is locally nilpotent, for any $x\in \mf m$. This leads to a well-defined functor $Q_\varphi\otimes_{\mc  W_\varphi}(\_):  W_\varphi\text{-Mod}\rightarrow  \mf a \text{-WMod}^\varphi$, which is a right adjoint of $\text{Wh}_\varphi(\_)$.

In what follows, we let $\{u_1,\ldots,u_m\}$ be a basis for $\mf m$ such that $\{u_1,\ldots,u_{m'}\}\subseteq \mf m_\oa$ and $\{u_{m'+1},\ldots, u_{m}\}\subseteq \mf m_\ob$ and they  are homogenous with respect to   the $\Z$-grading of $\mf m$. Let  
\begin{align}  
	&\{x_1,\ldots , x_{m'}\}\subseteq\mf a_\oa, ~\{x_{m'+1},\ldots, x_m\}\subseteq\mf a_\ob, \label{eq::552}
\end{align} be homogenous elements with respect to the $\Z$-grading such that $\deg(x_s)= -1+d_s$, where $d_s := -\deg(u_s)>0$, for any $1\leq s\leq m$.  
The following theorem is an analogue of the classical Skryabin equivalence \cite{Skr} in our setup. For the proof, see Appendix \ref{sect::append}. We refer to \cite[Theorem 22]{CC23_2} for an analogous version for quasi-reductive Lie superalgebras.
\begin{thm} \label{thm1} Retain the notations and assumptions above. Suppose that  the following conditions are satisfied: 
	\begin{itemize} 
		\item[(1)] For any $1\leq i\leq m$, we have  $[u_i,x_i]\in \mf m(-1)$ and $\varphi([u_i,x_i]) =1$.
		\item[(2)] Let  $1\leq i\neq j\leq m$. If  $[u_i, x_j]\in \mf m(-1)$  then $\varphi([u_i,x_j])=0$.
		\item[(3)] $\varphi$ vanishes on $\bigoplus_{i\leq -2}\mf a(i)$.
	\end{itemize}
 Then the Whittaker functor $\emph{Wh}_\varphi(\_):\mf a \emph{-WMod}^\varphi\rightarrow  \mc W_\varphi\emph{-Mod}$ is an equivalence of categories with the quasi-inverse $Q_\varphi\otimes_{\mc  W_\varphi}(\_): \mc  W_\varphi\emph{-Mod}\rightarrow \mf a \emph{-WMod}^\varphi$.
\end{thm}

\subsection{Finite SUSY $W$-algebras $\mc{SW}_c(\mf s,e)$}\label{sec:finite:SUSY} In the remainder of the paper, we let $\mf s$ be a basic classical Lie superalgebra. We shall keep the notation and assumptions of the previous subsection.  Let $e$ be an odd nilpotent element with $[e,e]=2E\not=0$ in $\mf s$. In the sequel, we shall always assume that $e$ is {\em neat} in the sense of \cite{ES22}. Recall that this means that, for any finite-dimensional $\mf s$-module $V$, the $\C e+\C E$-module $\text{Res}^{\mf s}_{\C e+\C E}V$ is a direct sum of indecomposable modules, each of which has nonzero super-dimension. According to \cite[Theorem 4.2.1]{ES22}, there exists a copy of $\mf{osp}(1|2)$ spanned by the basis elements $\{ F,f,h,e,E \}\subset\mf s$, and, additionally, all such copies obtained from $e$ are  conjugate to each other. A simple criterion to check whether an odd element is neat is provided by \cite[Lemma 6.5.10]{ES22}, which says that $e$ is neat is equivalent to the existence of a finite-dimensional faithful $\mf s$-module V such that the $\C e+\C E$-module $\text{Res}^{\mf s}_{\C e+\C E}V$ is a direct sum of indecomposable modules, each of which has nonzero super-dimension.

Suppose that $e$ is a neat element in $\mf s$, and we have a corresponding copy of $\mf{osp}(1|2)=\langle F,f,h,e,E\rangle\subset\mf s\subset\g$ as before. The adjoint action $\ad h: \mf g\rightarrow \g$ provides a $\Z$-grading $\g=\bigoplus_{i\in \Z} \g(i)$, where $\g(i) = \{x\in \g|~\ad h(x)=ix\}$.  Let $\mf s(i) = \g(i)\cap \mf s$, for all $i$. Then we have 
	\begin{align*} \g(i)= \left\{\begin{array}{ll}  \mf s(i)\oplus \ov{\mf s}(i), &  \text{ if }i\neq 0;\\ 	\mf s(0)\oplus \ov{\mf s}(0)+\C z, & \text{ if }i=0. 
	\end{array} \right. \end{align*} 
For any $i\in \Z$, we set $\mf s_{\leq i}:=\bigoplus_{j\leq i} \mf s(i)$ and  $\g_{\leq i}:=\mf s_{\leq i}\oplus \ov{\mf s}_{\leq i}$.  
Define the following $\Z$-graded subalgebra 
$$\mf m:= \mf g_{\leq -1} = \mf s_{\leq -1}\oplus \ov{\mf s}_{\leq -1}.$$ 

 Recall that $(\_|\_)$ denotes  a given even non-degenerate supersymmetric invariant bilinear form of $\mf s$, which induces a non-degenerate bilinear in $(\_|\_)'$ of $\mf s\otimes \Lambda(\theta)$ from Subsection \ref{sect::cen::212}.  There is a unique $\chi^e \in \mf m^\ast$ such that ($x\in \mf s_{\le -1}, i=0,1$) 	\begin{align*} \chi^e(x\otimes \theta^i)= (e|x\otimes \theta^i)' = \left\{\begin{array}{ll}  \mf (e|x), &  \text{ if } i= 1;\\ 0, & \text{ otherwise}. 
	\end{array} \right. \end{align*} 
 Since $e$ is an odd element lying in $\g(1)$, we see that $\chi^e$ vanishes on $[\mf m,\mf m]$, and therefore defines a character on $\mf m$. Set
\begin{align*}
    \mf m_{\chi^e}:=\{x-\chi^e(x)\mid x\in\mf m\}.
\end{align*}
As in Section \ref{sect::41}, we form the corresponding algebra 
 \begin{align*}
     &\mc W_{\chi^e}=\left(U(\g)/U(\g)\mf m_{\chi^e}\right)^{\text{ad}\mf m}.
 \end{align*}  
Let $U^c(\g)$ be the quotient algebra of $U(\g)$ by the ideal generated by $z-c$ and $Q_{\chi^e}^c$ be the quotient module $   Q_{\chi^e}/(z-c)Q_{\chi^e}$.  We recall the {\em finite supersymmetric (SUSY) $W$-algebra of $\mf s$ associated to $e$ at level $c$}:  \[\mc{SW}_c(\mf s,e) :=(Q_{\chi^e}^c)^{\ad{\mf m}}\cong (U^c(\g)\otimes _{U(\mathfrak{m})}\mathbb{C}_{\chi^e})^{\text{ad}\mathfrak{m}},\]
introduced in Subsection \ref{sect::111}. Observe that the canonical quotient $Q_{\chi^e}\rightarrow Q^c_{\chi^e}$ induces a homomorphism of algebras:
\begin{align}\label{eq:psie}
    \psi_e: \mc W_{\chi^e}\rightarrow \mc{SW}_c(\mf s,e),\qquad y+I_{\chi^e}\stackrel{\psi^e}{\rightarrow} y+\left(I_{\chi^e}+(z-c)U(\g)\right),
\end{align}
for $ y+I_{\chi^e}\in \mc W_{\chi^e}$. Similarly, we let $U^c(\g)$-WMod$^{\chi^e}$ denote the category of all  $U^c(\g)$-modules on which $x-\chi^e(x)$ acts locally nilpotently, for all $x\in\mf m.$ We may regard $U^c(\g)$-WMod$^{\chi^e}$ as the full subcategory of $\g$-WMod$^{\chi^e}$ consisting of objects on which $z-c$ acts trivially. Also, let $\mc{SW}_c(\mf s,e)\Mod$ denote the category of all $\mc{SW}_c(\mf s,e)$-modules. Furthermore, we define the associated Whittaker functor $$\text{Wh}^c_{\chi^e}(-): U^c(\g)\text{-WMod}^{\chi^e} \rightarrow \mc{SW}_c(\mf s,e) \Mod$$ in a completely analogous fashion as in \eqref{eq:Whitt:func}.

\begin{thm} Retain the notations and assumptions above. The Whittaker functors
\begin{align}
    &\emph{Wh}_{\chi^e}(-):\g \emph{-WMod}^{\chi^e}\rightarrow  \mc W_{\chi^e}\emph{-Mod},\label{eq::21}\\
    &\emph{Wh}^c_{\chi^e}(-): U^c(\g)\emph{-WMod}^{\chi^e}\rightarrow  \mc{SW}_c(\mf s,e)\emph{-Mod} \label{eq::22}
\end{align}
  are equivalences of categories with the respective quasi-inverses
  \begin{align}
      &Q_{\chi^e}\otimes_{\mc W_{\chi^e}}(-):  \mc W_{\chi^e}\emph{-Mod}\rightarrow  \g \emph{-WMod}^{\chi^e},\label{eq::23}\\
      &Q^c_{\chi^e}\otimes_{\mc{SW}_c(\mf s,e)}(-): \mc{SW}_c(\mf s,e)\emph{-Mod} \rightarrow U^c(\g)\emph{-WMod}^{\chi^e}. \label{eq::24}
  \end{align}
\end{thm}
\begin{proof} First, we shall use Theorem \ref{thm1} to give a proof of the equivalences \eqref{eq::21}, \eqref{eq::23}.
	Let $\{u_1,\ldots,u_m\}$ be a basis for $\mf m$ such that $$\{u_1,\ldots,u_{m'}\}\subseteq \mf m_\oa,~\{u_{m'+1},\ldots, u_{m}\}\subseteq \mf m_\ob,$$ are homogeneous with respect to   the $\Z$-grading. 
	It suffices to give a set of homogeneous elements $\{x_i\}_{i=1}^m$ as in \eqref{eq::552} which satisfies the Conditions $(1)$ and $(2)$ in Theorem \ref{thm1}. Since the adjoint action $\ad e: \mf m\rightarrow \mf s+\ov{\mf s}$ is injective, the vectors   $\{[e,u_i]\}_{i=1}^m$ are linearly independent. 
	
We note that the restriction $(\_|\_)': \mf s\otimes \Lambda(\theta)\times \mf s\otimes \Lambda(\theta)\rightarrow \C$ provides    non-degenerate parings $(\_|\_)'|_{\mf s(p)\times\ov{\mf s}(-p)}$ and $(\_|\_)'|_{ \ov{\mf s}(p)\times \mf s(-p)}$, respectively, for any $p\in \Z$ such that $\mf s(p)\neq 0$, and we have $(\mf s(p)\otimes \Lambda(\theta)|~\mf s(q)\otimes \Lambda(\theta))'=0$ for $p+q\neq 0$. This implies that  $\chi^e(\mf g(i)) =0$, for any $i\leq -2$. In addition, there exists a set of homogeneous elements $\{x_i\}_{i=1}^m\subseteq \mf s\otimes \Lambda(\theta)$ with respect to both $\Z_2$- and $\Z$-gradings such that ${\chi^e}([u_i,x_j]) = (e|[u_i,x_j])' = ([e,u_i]|x_j)' =\delta_{ij}$, for any $1\leq i,j\leq m$.  	Finally, we observe that $[u_i,x_i]\in   \mf m(-1)$. Therefore, the conclusion follows by \mbox{Theorem \ref{thm1}}.  

To establish equivalences \eqref{eq::22}, \eqref{eq::24}, it suffices to show that $\psi_e: \mc W_{\chi^e}\rightarrow \mc{SW}_c(\mf s,e)$ in \eqref{eq:psie} is an isomorphism. We note that the generalized Gelfand-Graev module $Q_{\chi^e}$ is an object in $\g$-WMod$^{\chi^e}$. Using the identifications $\mc W_{\chi^e} = \text{Wh}_{\chi^e}(Q_{\chi^e})$ and $\mc{SW}_c(\mf s,e) = \text{Wh}_{\chi^e}(Q^c_{\chi^e})$, the homomorphism $\psi_e$ can be identified as the image of the canonical quotient $Q_{\chi^e}\rightarrow Q^c_{\chi^e}$ under the Whittaker functor $\text{Wh}_{\chi^e}(\_)$. Since \eqref{eq::21} is an equivalence, we obtain a short exact sequence
\begin{align*}
&0\rightarrow \text{Wh}_{\chi^e}((z-c)Q_{\chi^e}) \rightarrow \mc W_{\chi^e} \xrightarrow{\psi_e} \mc{SW}_c(\mf s,e)  \rightarrow 0.
\end{align*} This implies that $\psi_e$ is surjective with kernel $\ker(\psi_e) =\text{Wh}_{\chi^e}((z-c)Q_{\chi^e}).$ Finally, it remains to prove that $\ker(\psi_e) = (z-c)\mc W_{\chi^e}$. To see this, we let $$\iota:  (z-c)\text{Wh}_{\chi^e}(Q_{\chi^e}) \hookrightarrow \text{Wh}_{\chi^e}( (z-c)Q_{\chi^e})$$ be the natural inclusion of $\mc W_{\chi^e}$-modules with cokernel $K$. Applying the functor $Q_{\chi^e}\otimes_{\mc W_{\chi^e}}(\_)$,  we have an  exact sequence of $\g$-modules:
\begin{align*}
&Q_{\chi^e}\otimes_{\mc W^{\chi^e}}(z-c)\text{Wh}_{\chi^e}(Q_{\chi^e}) \xrightarrow{\iota'}     Q_{\chi^e}\otimes_{\chi^e}  \text{Wh}_{\chi^e}((z-c)Q_{\chi^e}) \rightarrow Q_{\chi^e}\otimes_{\mc W_{\chi^e}} K\rightarrow   0
\end{align*} Using the equivalences in \eqref{eq::21} and \eqref{eq::23}, we conclude that the first and second terms above are naturally isomorphic to $(z-c)Q_{\chi^e}$ so that $\iota'$ turns out to be the identity map on this module. Consequently, we have $K=0$. This completes the proof.
\end{proof}

 Recall that an odd nilpotent element $e$ is called {\em principal nilpotent} in $\mf s$, if $E=e^2$ is even principal nilpotent in $\mf s_\oa$. It is well-known that, for a finite dimensional semisimple Lie algebra, a principal nilpotent element is characterized by the property that the dimension of its centralizer coincides with the rank of the Lie algebra. For an even principal nilpotent element $E$ in a basic Lie superalgebra that is not a Lie algebra this same property cannot characterize $E$. However, we have the following same characterization for a neat odd nilpotent element in a basic Lie superalgebra. Recall that the rank of a basic classical Lie superalgebra is the rank of its even subalgebra.

 \begin{prop}\label{prop:rk=dimcen}
Let $\mf s$ be a basic classical Lie superalgebra and suppose that $e$ is an odd neat nilpotent element in $\mf s$. Then $e$ is principal nilpotent if and only if the dimension of the centralizer of $e$ in $\mf s$ equals the rank of $\mf s$.
\end{prop} 

\begin{proof}
    Suppose that $e$ is principal nilpotent. Then $E=e^2$ is principal nilpotent in $\mf s_\oa$, and hence the centralizer of $E$ in $\mf s_\oa$ equals the rank of $\mf s_\oa$. Since $e$ is neat, we have an embedding of $\mf{osp}(1|2)=\langle E, e, h, f, F\rangle$ inside $\mf s$. From the classification of such embeddings of $\mf{osp}(1|2)$ into $\mf s$ in \cite[Section 10]{FRS92} we only have the following possibilities for $\mf s$: $\gl(n|n+1)$, $\gl(n+1|n)$, $\mf{sl}(n|n+1)$, $\mf{sl}(n+1|n)$, $\mf{osp}(2n\pm1|2n)$, $\mf{osp}(2n|2n)$, $\mf{osp}(2n+2|2n)$, $n\ge 1$, and $D(2,1|\alpha)$. Now, recall that every finite dimensional $\mf{osp}(1|2)$-module is completely reducible, and furthermore every finite-dimensional non-trivial irreducible module restricted to $\mf a=\langle E, h, F\rangle\cong \mf{sl}(2)$ is a direct sum of two irreducible modules of opposite parity with highest weights differing by $1$. Let us denote the irreducible $\mf a$-module of highest weight $k\ge 0$ by $L(k)$. Since $E$ is principal nilpotent in $\mf s_\oa$, $\mf s_\oa$ decomposes into $\text{rank}\mf s_\oa$  irreducible $\mf a$-modules. Hence to prove that the centralizer of $e$ in $\mf s$ equals the rank of $\mf s_\oa$, it suffices to prove that $\mf s_\ob$, as an $\mf a$-module, does not contain any trivial $\mf a$-submodule.
    
    Without loss of generality, we may assume that $h=2\rho^\vee_\oa$ since $E$ is principal nilpotent. Using this, we can easily compute the $\mf a$-highest weights of the natural modules of the simple Lie algebra of types $A,B,C,D$ at $2\rho^\vee_\oa$. From this calculation, we conclude that  we have the following isomorphism of $\mf a$-modules: For $\mf s$ of type $A$ in the list above $\mf s_\ob\cong\left( L(n)\otimes L(n-1)\right)\oplus \left(L(n)\otimes L(n-1)\right)$. For $\mf s=\mf{osp}(2n\pm 1|2n)$, we have $\mf s_\ob\cong L(2n\pm 1-1)\otimes L(2n-1)$, while for $\mf s=\mf{osp}(2n+2|2n)$ and $\mf s=\mf{osp}(2n|2n)$,  we have $\mf s_\ob\cong \left(L(2n)\oplus L(0)\right)\otimes L(2n-1)$ and $\mf s_\ob\cong \left(L(2n-2)\oplus L(0)\right)\otimes L(2n-1)$, respectively. Finally, for $\mf s=D(2|1,\alpha)$ we have $\mf s_\ob\cong \left( L(2)\oplus L(0)\right)\otimes L(1)$. It follows that $\mf s_\ob$ has no $\mf a$-invariants, and hence the dimension of the centralizer of $e$ in $\mf s$ equals the rank of $\mf s$.

    On the other hand, if the dimension of the centralizer of $e$ in $\mf s$ equals the rank of $\mf s_\oa$, then by the finite-dimensional representation theories of $\mf{osp}(1|2)$ and $\mf{sl}(2)$, the dimension of the centralizer of $E$ on $\mf s_\oa$ is at most the rank of $\mf s_\oa$. But then the dimension of the centralizer of $E$ on $\mf s_\oa$ must be exactly the rank of $\mf s_\oa$ and so $E$ must be principal nilpotent in $\mf s_\oa$. This, however,  means precisely that $e$ is principal nilpotent in $\mf s$.
\end{proof}

\subsection{Proof of Theorem \ref{thm::2}} \label{sect::43}
Now, suppose that $e$ is a neat principal nilpotent element in $\mf s$. In this case, the subalgebra $\mf m$ from Subsection \ref{sec:finite:SUSY} is of the form $\mf m=\mf n\oplus \ov{\mf n}$ for some nilradical $\mf n$ of a Borel subalgebra $\mf b$ of $\mf s$. Set  $\chi$ to be the restriction of $\chi^e$ to $\hat{\mf n}_\oa$. Recall the associated character $\zeta = \zeta_\chi: \mf n_\oa\rightarrow \C$  
 from Subsection \ref{sect::111}:
 \begin{align*}  &\zeta(X):=  \left\{\begin{array}{ll} \frac{-(\alpha_1|\alpha_2)}{c} (e|X_1)(e|X_2), \\ \text{\qquad if $X=[X_1,X_2]$ for some $X_i\in \mf s^{\alpha_i}$ with $\alpha_i \in \Delta_\ob$};\\ 	0,  ~~~ \text{otherwise}. 
	\end{array} \right.   \end{align*}
 Then for any  non-zero complex number $c$, the category  $\g\text{-Wmod}^{\chi}_c$ is equal to the full subcategory of $\g\text{-WMod}^{\chi^e}$ consisting of finitely-generated $\g$-modules which are annihilated by $z-c$. To see this, let $M\in \g\text{-Wmod}^{\chi}$ and $X\in \mf m_\ob$. Since $\chi^e$ is a character of $\mf m$,  it follows that  $2X^2= [X,X] = [X,X]-\chi([X,X])$ acts on $M$ locally nilpotently.  This proves the claim. Recall that $\mc{SW}_c(\mf s,e)$-mod denotes the category of all finitely generated $\mc{SW}_c(\mf s,e)$-modules. Then, combined with Theorem \ref{thm::1}, we obtain Theorem \ref{thm::2} which we restate below.
\begin{thm} \label{thm::2::restt}
	The following functor defines an equivalence of categories 
		\begin{align*}
		&\emph{Wh}_{\chi^e} \circ  (\_\otimes \mf F_c^\eta): \mf s\text{-}\emph{Wmod}^{\zeta}\rightarrow \mc{SW}_e(\mf s,e)\mod,
		\end{align*}  with quasi-inverse  $  \Hom_{\mf c}(\mf F_c^\eta,\_) \circ Q_{\chi^e}$.  
\end{thm} 
The principal finite $W$-superalgebra  $U(\mf s,E)$ associated to the even nilpotent element $E$ can be defined in the same fashion (see, e.g., \cite[Subsection 4.1]{CC23_2}), that is,  $U(\mf s,E)$ is the opposite algebra of the endomorphism algebra of the generalized Gelfand-Graev $\mf s$-module $U(\mf s)\otimes_{U(\mf n)} \C_{\chi^E}$, where $\chi^E$ is the following analogous nilcharacter: $\chi^E(\_):=(E|\_):\mf n\rightarrow \C$.

Theorem \ref{thm::2::restt} implies the following direct connection between the representation theories of the two $W$-superalgebras $\mc{SW}_c(\mf s,e)$ and $U(\mf s,E)$. 
\begin{cor}\label{Cor::23}
For each non-zero $c$, the category $\mc{SW}_c(\mf s,e)$-mod is equivalent to the category of finitely-generated modules over the principal finite $W$-superalgebra $U(\mf s, E)$. In particular, $\mc{SW}_c(\mf s,e)\text{-mod}\cong\mc{SW}_{c'}(\mf s,e)\text{-mod}$, for $c,c'\not=0$.
\end{cor}
\begin{proof}
Recall that we fix root vectors $E_\alpha, F_\alpha\in \mf s^\alpha$, for $\alpha\in \Phi^+$. Since $e$ is principal nilpotent, we can write $e=\sum_{\alpha\in \Delta_\ob} c_\alpha F_\alpha$, for some non-zero scalars $c_\alpha$ and $E$ is principal nilpotent. It follows by  \cite[Proposition 36]{CC23_2} and Theorem \ref{thm::2::restt}  that we are left to show that $\zeta: \mf n_\oa\rightarrow \C$ is a regular character of $\mf n_\oa$ in the sense of Kostant \cite{Ko78}, that is, $\zeta$ does not vanish on any simple root vector in $\mf n_\oa$. To see this, let $\alpha_1, \alpha_2\in \Delta_\ob$ be such that $\alpha_1+\alpha_2$ is also a root. We calculate \begin{align*}
&\zeta([E_{\alpha_1}, E_{\alpha_2}]) = \frac{-(\alpha_1|\alpha_2)}{c}(c_{\alpha_1}F_{\alpha_1}|E_{\alpha_1})(c_{\alpha_2}F_{\alpha_2}|E_{\alpha_2})= \frac{-(\alpha_1|\alpha_2)c_{\alpha_1}  c_{\alpha_2}}{c} \neq 0.
\end{align*} From the classification of embeddings of $\mf{osp}(1|2)$ into a basic classical Lie superalgebra such that $e$ is principal nilpotent in \cite[Section 10]{FRS92}, we conclude that in these cases every even root which is simple in $\Phi_\oa^+$ is either a sum of two different isotropic odd simple roots or else is twice the same non-isotropic simple odd root. Consequently, $\zeta$ is regular. This completes the proof.
\end{proof}

\appendix
\section{Proof of Theorem \ref{thm1}} \label{sect::append}
This section is devoted to providing a proof of Theorem \ref{thm1}, which is an adaptation of Skryabin's original proof in \cite{Skr} to our setting; see also \cite[Section 5]{CC23_2} for a similar version of  finite $W$-(super)algebras associated with  Lie superalgebras from Kac's list \cite{K77}.  In this subsection, we retain the notations and assumptions of Subsection \ref{sect::41}.

 For  ${\bf a} = (a_1,\ldots,a_m)\in X:=\mathbb Z_{\geq 0}^{m'}\times \{0,1\}^{m-m'}$, we define vector $u^{\bf a}\in U(\mf m)$: \begin{align*} 	&u^{\bf a} = (u_1-\varphi(u_1))^{a_1}(u_2-\varphi(u_2))^{a_2}\cdots (u_{m'}-\varphi(u_{m'}))^{a_{m'}} u_{m'+1}^{a_{m'+1}}\cdots u_{m}^{a_m}. \end{align*}

Furthermore, define the following notations 
\begin{align*}
	&|{\bf a}| := \sum_{s} a_s,~\wt{\bf a }:= \sum_{s}d_sa_s\in \Z_{\geq 0},~\text{and } x^{\bf a} = x_1^{a_1}x_2^{ a_2}\cdots x_m^{ a_m}\in U(\mf a).  
\end{align*}
Consider any linear ordering $<$ on $X$ subject to the condition \begin{align*}
	&{\bf a}< {\bf b} \text{ whenever either }\wt {\bf a} <\wt {\bf b}  \text{ or } \wt {\bf a} = \wt {\bf b}, |{\bf a}| >|{\bf b}|.
\end{align*}
Our first goal is to show the following lemma:

\begin{lem} \label{lem::3}Suppose that Conditions $(1),(2)$ and $(3)$ in Theorem \ref{thm1} are satisfied. Then for any $M\in  \mf a$-\emph{Wmod}$^\varphi$, we have 
 \begin{align}
	&u^{\bf a} x^{\bf a} v =cv,\text{ for some non-zero  scalar $c$;}\label{eq::66}\\
	&u^{\bf a}x^{\bf b} v= 0, \text{ when }{\bf a} > {\bf b}, \label{eq::67}
\end{align}
for any $v\in {\emph{Wh}}_\varphi(M)$.  
\end{lem}
\begin{proof} First, we define $M_{0,0}:=\text{Wh}_\varphi(M)$. 
For any $i, j\geq 0$, we also define $M_{i,j}\subseteq M$ as the subspace spanned by all elements $r_1r_2\cdots r_\ell v$ with $v\in \text{Wh}_\varphi(M)$ with
\begin{align}
	&\ell\geq j,~r_1\in \mf a({-1+i_1}), r_2\in \mf a({-1+i_2}), \ldots, r_\ell \in \mf a({-1+i_\ell}), \label{eq::68}\\ &\text{where } i_1, i_2,\ldots, i_\ell>0 \text{ such that } i_1+i_2+\cdots +i_\ell\leq i.\label{eq::69}
\end{align} Define $M_{i,j} = M_{i,0}$ for any $i\geq 0$ and $j<0$.  Let $M_{i,j} = 0$ for any $i<0$. 	Also, we define $M_{0,j}=0$, for $j>0$. By definition, we have $M_{i,j}\subseteq M_{i',j'}$ and $\g({-1+c})M_{i,j} \subseteq M_{i+c,j+1}$, for any $i\leq i'$, $j\geq j'$ and $c>0$.

 Let $y\in \mf m(-d)$ with $d\geq  1$ and $y\in {\mf m}_{\oa}\cup {\mf m}_\ob$. Let  $r_1,\ldots, r_\ell\in \mf a$ be homogeneous elements with respect to both $\Z_2$-grading and $\Z$-grading of $\mf a$. Suppose $v\in \text{Wh}_\varphi(M)$ and $r_1,\ldots, r_\ell, j, \ell$ satisfy the conditions \eqref{eq::68} and \eqref{eq::69}. We claim that: 
 
 {\bf Claim 1}:  {\em There is an element $R\in M_{i-d,j}+M_{i-d-1,0}$ \emph{(}depending on $y,r_1\ldots, r_\ell,v$\emph{)} such that
 \begin{align}
 	&(y-\varphi(y))r_1r_2\cdots r_\ell v = \sum_{1\leq s\leq \ell,~i_s=d} c_s r_1\cdots \widehat{r_s} \cdots r_\ell [y,r_s]v+R,\label{eq::pnSk}\\
 	&(y-\varphi(y))r_1r_2\cdots r_\ell v\in M_{i-d,j-1}+M_{i-d-1,0},\label{eq2::pnSk}
 \end{align}   where the notation $\widehat{r_s}$, as usual, denotes omission of $r_s$, and \[c_s = (-1)^{ p(y)(p({r_1})+\cdots+p({r_{s-1}})+p({r_{s+1}})+\cdots+p({r_\ell}))+p({r_s})(p({r_{s+1}})+\cdots +p({r_\ell}))}.\]} 

 The Claim 1 is an analogue of Claim 3 in the proof of \cite[Theorem 1.3]{Skr2}. We note that the left side in the first equation in \eqref{eq::pnSk} can be  
 written as $(y-\varphi(y))r_1r_2\cdots r_\ell v = [y, r_1r_2\cdots r_\ell]v$ since $yv=\varphi(y)v$ and  $\varphi(\mf m_\ob)=0$. Therefore, to prove \eqref{eq::pnSk} and \eqref{eq2::pnSk}, we shall show that 
 \begin{align}
 &[y, r_1r_2\cdots r_\ell]v =  \sum_{1\leq s\leq \ell,~i_s=d} c_s r_1\cdots \widehat{r_s} \cdots r_\ell [y,r_s]v+R,\label{eq::pnSk2} \\
 &[y, r_1r_2\cdots r_\ell]v= M_{i-d,j-1}+M_{i-d-1,0} \subseteq M_{i-d,0}, \label{eq::9}
 \end{align} for some $R\in M_{i-d,j}+M_{i-d-1,0}$.  
We adapt the arguments in \cite[Thoerem 1.3]{Skr2} to give a proof. First,  we note that $\deg[y,r_s] = -1+i_s-d$.  If  $i_s>d$ then 
\begin{align*}
&r_1\cdots r_{s-1}[y,r_s] r_{s+1}\cdots r_{\ell}v \in M_{i-d,j}.
\end{align*}
Suppose that   $i_s\leq d$. In this case we note that $\deg[y,r_s] <0$.
By induction on $\ell$ in \eqref{eq::9} we get 
\begin{align*}
&r_1\cdots r_{s-1}[[y,r_s],r_{s+1}\cdots r_{\ell}]v \in M_{i_1+\cdots+i_{s-1}+i_{s+1}\cdots+i_\ell  +(-2+i_s-d),0}\subseteq  M_{i -d-1,0}, 
\end{align*} namely, the following element 
\begin{align*}
	&r_1\cdots r_{s-1}[y,r_s]r_{s+1}\cdots r_{\ell}v - (-1)^{(p(y)+p({r_s}))(p({ r_{s+1}})+\cdots+p(r_\ell))} r_1\cdots \hat r_{s}\cdots r_{\ell} [y,r_s] v
\end{align*} lies in $  M_{i -d-1,0}$.

 If $i_s<d$ then $[y,r_s]v = \varphi([y,r_s])v=0$ since $[y,r_s]\in \bigoplus_{i\leq -2}\mf a(i)$. If $i_s=d$ then $[y,r_s]v\in \text{Wh}_\varphi(M)$, and so this case leads to $c_s r_1\cdots \hat r_s \cdots r_\ell [y,r_s]v\in M_{i-d,j-1}$.  This proves the Claim 1.

    Let ${\bf a, b}\neq {\bf 0}$. By Claim 1 we have \[u^{\bf a}M_{i,j} \subseteq M_{i-\wt{\bf a}, j-|{\bf a}|}+M_{i-\wt{\bf a}-1},0,\] for any ${\bf a}\in X$. Now, suppose that $\wt{\bf b}=i$ and $|{\bf b}|=j$, then we have  $x^{\bf b}v\in M_{i,j}$. It follows that  $u^{\bf a}x^{\bf b}v=0$ whenever either $\wt{\bf a} >\wt{\bf b}$ or $\wt{\bf a}=\wt{\bf b}$ and $|{\bf a}|<|{\bf b}|$. This proves the assertion \eqref{eq::67}.

    	Let $\wt{\bf a}=\wt{\bf b}=i$ and $|{\bf a}| =|{\bf b}|=j$. Our final goal is to show that $u^{\bf a}x^{\bf b} v=0$ for ${\bf a}\neq {\bf b}$ and $u^{\bf a}x^{\bf a} v=cv$, for some non-zero $c\in \C$. We proceed by induction on $j$. Assume that $j>0$ and the assertions hold for smaller values of $j$. Define $p$ to be such that $a_p\neq 0$ and $a_s=0$ for any $p<s\leq m$. Denote by ${\bf e}_p$ the $m$-tuple with $1$ at the $p$-th position and $0$ elsewhere. Then   $u^{\bf a} = u^{{\bf a}-{\bf e}_p}(u_p-\varphi(u_p))$. 
  Recall that $d_p =-\deg(u_{p})$.   Observe that $u^{{\bf a}-{\bf e}_p}M_{i-d_p,j} = 0$ (since $\wt(u^{{\bf a}-{\bf e}_p})=i-d_p$ with $|u^{{\bf a}-{\bf e}_p}| =j-1 <j$) and $u^{{\bf a}-{\bf e}_p}M_{i-d_p-1,0}=0$ (since $\wt(u^{{\bf a}-{\bf e}_p})=i-d_p > i-d_p-1$). Also, we recall that $\deg(x_s) =-1+d_s$, for any $1\leq s\leq m$. By Equations \eqref{eq::pnSk}-\eqref{eq2::pnSk} it follows that
  \begin{align*}
  	&u^{\bf a}x^{\bf b}v =  u^{{\bf a}-{\bf e}_p} \sum_{1\leq s\leq m,~d_s=d_p} c_s x^{{\bf b}-{\bf e}_s} [u_p,x_s]v.
  \end{align*} 
Since $d_s=d_p$ in the summation above implies that $[u_p,x_s]\in \mf m(-1)$, we have    \begin{align*}
	&u^{\bf a}x^{\bf b}v  =\sum_{1\leq s\leq m,~d_s=d_p}c_s u^{{\bf a}-{\bf e}_p}x^{{\bf b}-{\bf e}_s}\varphi([u_p,x_s])v.
\end{align*} By assumption, we have that $\varphi([u_p,x_s])\neq 0 \Leftrightarrow p=s$. This implies that $u^{\bf a}x^{\bf b}v=0$ provided that $b_p=0$. In the case that $b_p>0$, we have 
 \begin{align*}
	&u^{\bf a}x^{\bf b}v  =c_p  b_pu^{{\bf a}-{\bf e}_p}x^{{\bf b}-{\bf e}_p}\varphi([u_p,x_p])v=c_p  b_pu^{{\bf a}-{\bf e}_p}x^{{\bf b}-{\bf e}_p}v,
\end{align*} where $c_p =\pm 1$ depends on the parities   $x_1^{b_1}x_2^{b_2}\cdots x_{p-1}^{b_{p-1}}$ and $u_p$ as described in Equations \eqref{eq::pnSk}-\eqref{eq2::pnSk}.

   By induction hypothesis, $u^{\bf a}x^{\bf b} v=0$ unless ${\bf a}= {\bf b}$. If ${\bf a} ={\bf b}$, then by induction hypothesis again we get a non-zero vector $u^{{\bf a}-{\bf e}_p}x^{{\bf b}-{\bf e}_p}v \in \C v$. The conclusion in \eqref{eq::66} follows as well. This completes the proof of Lemma \ref{lem::3}. 
\end{proof}

 Retain the notations and assumptions above. The next  goal is to establish Lemmas \ref{lem::6} and \ref{lem::7} below. 

For each ${\bf a}\in X$, we let $I_{\bf a}$ be the ideal of $U(\mf m)$ spanned by $u^{\bf b}$ for $\bf b>\bf a$. 	Let $M\in $ $\mf a$-WMod$^\varphi$ and set $V:=\text{Wh}_\varphi(M)$. Then $\Hom_\C(U(\mf m),V)$ is a left ($\Z_2$-graded) $U(\mf m)$-module. Define  $E:=E_{V,\varphi}\subset \Hom_\C(U(\mf m),V)$ as the $\mf m$-submodule of all linear maps $f: U(\mf m)\rightarrow V$ such that $f(I_{\bf a})=0$, for some ${\bf a}\in X$. Let $\kappa': M\rightarrow V$ be an even linear map such that $\kappa'$ restricts to the identity map on $V$. Then the map 
\[\kappa: M\rightarrow E_{V,\varphi},~\kappa(a)(u) = (-1)^{p(a) p(u)} \kappa'(ua),~\text{for homogenous }u\in U(\mf m),~a\in M,\] defines a homomorphism of $\mf m$-modules. 
\begin{lem}  \label{lem::4}
 $\kappa$ is an isomorphism of $\mf m$-modules.
\end{lem}
\begin{proof} 
	First, we note that $\ker(\kappa)\in \mf a$-WMod$^\varphi$. If $\ker(\kappa)$ is non-zero, then $V\cap \ker(\kappa)\neq 0$, and so $\kappa(V\cap \ker(\kappa))\neq 0,$ a contradiction. This proves that $\kappa$ is injective. 
	
  It remains to show that $\kappa$ is surjective. 	Define $E^{\bf a}_{V,\varphi}:= \{f\in E_{V,\varphi}|~f(I_{\bf a})=0\}$, for any ${\bf a}\in X$. We certainly have $\kappa(M)\supseteq \kappa(V)= E_{V,\varphi}^{\bf 0}$.  We shall proceed by induction on the linear order $<$ on $X$. Suppose that $\kappa(M)\supseteq E_{V,\varphi}^{\bf a}$, for some ${\bf a}\in X$. Let ${\bf a}^+>{\bf a}$ be the successor of ${\bf a}$. Let $f\in E_{V,\varphi}^{\bf a^+}$, then there exists $d\in \C$ such that $f-\kappa(dX^{{\bf a}^+}f(u^{{\bf a}^+}))$ lies in $E_{V,\varphi}^{\bf a}$ by Lemma \ref{lem::3}. We may conclude that $\kappa(M)$ contains $E_{V,\varphi}^{{\bf a}^+}$. This completes the proof.
\end{proof}

\begin{rem}  By the proof of Lemma \ref{lem::4}, it follows that $E_{V,\varphi}^{\bf a}$ is spanned by $\{\kappa(X^{\bf b}v_{\bf b})|~{\bf b}\leq {\bf a},~v_{\bf b}\in \text{Wh}_\varphi(M)\}$, for any ${\bf a}\in X$.
\end{rem}

\begin{lem} \label{lem::6}
 Every element of $M$ can be uniquely expressed in the form $\sum_{\bf a}x^{\bf a}v_{\bf a}$, for some (finitely many) $v_{\bf a}\in \emph{Wh}_\varphi (M)$.
\end{lem}
\begin{proof}
	Suppose that $x^{\bf a}v_{\bf a}+\sum_{\bf b<\bf a}x^{\bf b}v_{\bf b}=0$, for some ${\bf a}\in X$ and $v_{\bf a}, v_{\bf b}\in \text{Wh}_\varphi (M)$. By Lemma \ref{lem::3} we have  $$\kappa(\sum_{\bf b\leq \bf a}x^{\bf b}v_{\bf b})(  u_{\bf a}) = \pm \sum_{\bf b\leq \bf a} {u}_{\bf a} x^{\bf b}v_{\bf b} = cv_{\bf a}\neq 0,$$ for some $c\in \C$. By Lemma \ref{lem::4}, it follows that $\{\kappa(x^{\bf a}v) |~{\bf a}\in X,~v\in \text{Wh}_\varphi(M)\}$ spans $E_{V,\varphi}$. This completes the proof.
\end{proof}

Recall that $Q_\varphi$ denotes the generalized Galfand-Graev module $U(\mf a)/I_\varphi$.  Let $1_\varphi =1_{U(\mf a)}+I_{\varphi}\in Q_\varphi$. Then we have the following:
\begin{lem} \label{lem::7}
 $Q_\varphi$ is a free (right) $\mc W_\varphi$-module with basis $\{x^{\bf a}1_\varphi\in X|~{\bf a}\in X\}$.
\end{lem}
\begin{proof}
The proof follows by an argument used in the proof of \cite{Skr}. Let ${\bf a}\in X$. Recall that $I_{\bf a}$ denotes the ideal of $U(\mf m)$ spanned by elements $u^{\bf b}$ for ${\bf b}\in X$ with ${\bf b}>{\bf a}.$ Define $Q_\varphi^{\bf a}:=\{v\in Q_\varphi|~I_{\bf a}v=0\}$. Since $\mc W_\varphi$ is isomorphic to $\End_{U(\mf a)}(Q_\varphi)^{\text{op}}$, it follows that  $Q_\varphi$ is a right $\mc W_\varphi$-module in a natural way. By Lemma \ref{lem::3}, we can define a homomorphism of right $W_\varphi$-modules $i: Q_\varphi^{\bf a} \rightarrow \mc W_\varphi$ by $i(v): = u^{\bf a}v$, for $v\in Q^{\bf a}_\varphi$. If we let ${\bf a}'$ be the predecessor of ${\bf a}$, then $\ker(i) = Q_\varphi^{\bf a'}$ by definition. Let $v\in \text{Wh}_\varphi(Q_\varphi)$. Then $v$ is a scalar multiple of $i(X^{\bf a}v)$ by Lemma \ref{lem::3}. This proves that $i$ is surjective, and then we obtain an isomorphism of  right $\mc W_\varphi$-modules $$Q_\varphi^{\bf a}/Q_\varphi^{\bf a'}\xrightarrow{\cong} \mc W_\varphi,~x^{\bf a}v_{\bf a}+\sum_{\bf b<\bf a}x^{\bf b}v_{\bf b}\mapsto c({\bf a})1_\varphi,$$ where $c({\bf a})\in \C$ is determined by $c({\bf a})v_{\bf a}= u^{\bf a}x^{\bf a}v_{\bf a}$. By induction on the linear order $<$ on $X$, it follows that  $Q_\varphi^{\bf a}\cong Q_\varphi^{\bf a'}\oplus \mc W_\varphi$ is a free right $\mc W_\varphi$-module with basis  $\{x^{\bf b}1_\varphi\in X|~{\bf b} <{\bf a}\}$.   \end{proof}

We are now ready to prove Theorem \ref{thm1}. 
\begin{proof}[Proof of Theorem \ref{thm1}]
Let $M\in$ $\mf a$-WMod$^\varphi$ and $V:= \text{Wh}_\varphi(M)$. Our first goal is to show that  the following map $\mu$ is an isomorphism: 
\begin{align}
&\mu: Q_\varphi\otimes_{\mc W_\varphi}V\rightarrow M,~\text{via} ~ \pr(u)\otimes v \mapsto uv, \text{~for $u\in U(\mf a)$, $v\in V$.}
 \end{align} 	By Lemma \ref{lem::7}, it follows that every element in $Q_\varphi\otimes_{\mc W_\varphi}V$ can be uniquely written as a finite sum $\sum_{\bf a}x^{\bf a}1_\varphi \otimes v_{\bf a}$, for ${\bf a}\in X$ and $v_{\bf a}\in \text{Wh}_\varphi(M)$. The bijectivity of $\mu$ now follows by Lemma \ref{lem::6}.

Let $V$ be a $\mc W_\varphi$-module. 
Our second goal is to show that the following homomorphism \[\nu: V \rightarrow   \text{Wh}_\varphi(Q_\varphi\otimes_{\mc W_\varphi}V),~\text{via } v\mapsto \pr(1_{U(\mf a)})\otimes v, \text{ for $v\in V$,}\] is an isomorphism. Since $Q_\varphi$ is a free $\mc W_\varphi$-module by Lemma \ref{lem::7}, $\nu$ is a monomorphism and so it remains to show that $\nu$ is surjective. To see this,  suppose on the contrary that there is ${\bf 0}\neq {\bf a}\in X$ such that  $$v = x^{\bf a}\otimes v_{\bf a}+\sum_{\bf b<\bf a}x^{\bf b}\otimes v_{\bf b}\in \text{Wh}_\varphi(Q_\varphi\otimes_{\mc W_\varphi}V),$$ then by Lemma \ref{lem::3} we get $ 0=   u^{\bf{a}} x^{\bf a}\otimes v_{\bf a}= c (1\otimes v_{\bf a})$ with $c\neq 0$, a contradiction. Therefore, we have $\nu(V) =  \text{Wh}_\varphi(Q_\varphi\otimes_{\mc W_\varphi}V)$.  This completes the proof. 
\end{proof}



\end{document}